\documentclass[11pt,a4paper,reqno,dvipdfmx]{amsart}

\usepackage[truedimen,margin=25truemm]{geometry} 
\usepackage{amsmath,amssymb,amsthm,amscd}
\usepackage{extarrows,mathtools}

\title[Cancellation-free quantum $K$-theoretic divisor axiom]
{Cancellation-free version of the quantum $K$-theoretic divisor axiom 
for the flag manifold in the quasi-minuscule case}
\author[R.~Kato]{Ryo Kato}
\address[Ryo Kato]{Doctoral Program in Mathematics, 
Degree Programs in Pure and Applied Sciences, 
Graduate School of Science and Technology, 
University of Tsukuba, 
1-1-1 Tennodai, Tsukuba, Ibaraki 305-8571, Japan.}
\email{s2430043@u.tsukuba.ac.jp}

\author[D.~Sagaki]{Daisuke Sagaki}
\address[Daisuke Sagaki]{Department of Mathematics, 
Institute of Pure and Applied Sciences, University of Tsukuba, 
1-1-1 Tennodai, Tsukuba, Ibaraki 305-8571, Japan.}
\email{sagaki@math.tsukuba.ac.jp}


\usepackage[dvipdfmx]{xcolor}

\allowdisplaybreaks
\numberwithin{equation}{section}

\newcommand{\Fg}{\mathfrak{g}}

\newcommand{\Ft}{\mathfrak{t}}

\newcommand{\BZ}{\mathbb{Z}}

\newcommand{\BC}{\mathbb{C}}

\newcommand{\CO}{\mathcal{O}}

\newcommand{\vpi}{\varpi}

\newcommand{\ba}{\mathbf{a}}
\newcommand{\bd}{\mathbf{d}}
\newcommand{\be}{\mathbf{e}}

\newcommand{\bp}{\mathbf{p}}
\newcommand{\bq}{\mathbf{q}}

\newcommand{\bt}{\mathbf{t}}
\newcommand{\bv}{\mathbf{v}}

\DeclareMathOperator{\wt}{wt}
\DeclareMathOperator{\qwt}{qwt}
\DeclareMathOperator{\ed}{end}
\DeclareMathOperator{\codim}{codim}

\newcommand{\af}{\mathrm{af}}

\newcommand{\pt}{\mathrm{pt}}
\newcommand{\Hom}{\mathrm{Hom}}
\newcommand{\QBG}{\mathrm{QBG}}
\newcommand{\QLS}{\mathrm{QLS}}
\newcommand{\Lie}{\mathrm{Lie}}

\newcommand{\edge}[1]{ \xrightarrow{\hspace{2pt}#1\hspace{2pt}} }
\newcommand{\Qe}[1]{ \xrightarrow[\mathsf{q}]{\hspace{2pt}#1\hspace{2pt}} }
\newcommand{\Be}[1]{ \xrightarrow[\mathsf{B}]{\hspace{2pt}#1\hspace{2pt}} }

\newcommand{\lng}{w_{\circ}}
\newcommand{\mcr}[1]{\lfloor #1 \rfloor}
\newcommand{\mxcr}[1]{\lceil #1 \rceil}

\newcommand{\ti}[1]{\widetilde{#1}}

\newcommand{\ol}[1]{\overline{#1}}

\newcommand{\pair}[2]{\langle #1,\,#2 \rangle}

\newcommand{\tb}[1]{\le_{#1}}
\newcommand{\tbmin}[3]{\min(#1W_{#2},\le_{#3}\nobreak)}

\newcommand{\J}{J}

\newcommand{\WJs}{W_{\J}}
\newcommand{\WJ}{W^{\J}}
\newcommand{\DJs}{\Delta_{\J}^{+}}
\newcommand{\DJp}{\Delta^{+} \setminus \Delta_{\J}^{+}}
\newcommand{\DJsm}{\Delta_{\J}^{-}}
\newcommand{\DJm}{\Delta^{-} \setminus \Delta_{\J}^{-}}

\newcommand{\QJ}{Q_{\J}}
\newcommand{\QJv}{Q_{\J}^{\vee}}
\newcommand{\QJvp}{Q_{\J}^{\vee,+}}

\newcommand{\bQBG}[1]{\mathbf{QBG}_{#1}^{\lhd}}
\newcommand{\bQLS}[1]{\mathbf{QLS}_{#1}^{\lhd}}
\newcommand{\bR}[1]{\mathbf{R}_{#1}^{\lhd}}
\newcommand{\hbR}[1]{\widehat{\mathbf{R}}_{#1}^{\lhd}}

\newcommand{\IL}[1]{ \iota_{\mathrm{L}}(#1) }

\theoremstyle{plain}
\newtheorem{lem}{Lemma}[section]
\newtheorem{prop}[lem]{Proposition}
\newtheorem{thm}[lem]{Theorem}
\newtheorem{cor}[lem]{Corollary}
\newtheorem{conj}[lem]{Conjecture}

\newtheorem{ithm}{Theorem}

\newtheorem{iconj}[ithm]{Conjecture}

\theoremstyle{definition}
\newtheorem{dfn}[lem]{Definition}

\newtheorem{idfn}[ithm]{Definition}

\theoremstyle{remark}
\newtheorem{ex}[lem]{Example}
\newtheorem{rem}[lem]{Remark}
\newtheorem{claim}{Claim}[lem]

\newtheorem{irem}[ithm]{Remark}

\newcommand{\bqed}{\quad \hbox{\rule[-0.5pt]{3pt}{8pt}}}

\newenvironment{enu}{%
 \begin{enumerate}%
}{\end{enumerate}}

\begin{document}


%
\begin{abstract}
We prove a cancellation-free version of the quantum $K$-theoretic divisor axiom 
for the flag manifold in the quasi-minuscule case under some assumptions.
Namely, we remove the cancellations 
from the quantum $K$-theoretic divisor axiom obtained in \cite{LNSX} in the case where 
the fundamental weight corresponding to the divisor class is quasi-minuscule. 
\end{abstract}

\maketitle

%
\section{Introduction.}

Let $G$ be a connected, simply-connected, simple (linear) algebraic group over $\BC$
with $T$ a maximal torus of $G$, and $B$ a Borel subgroup of $G$, and set 
$\Fg := \Lie(G)$ and $\Ft := \Lie(T)$, where 
$\Fg$ is a finite-dimensional simple Lie algebra over $\BC$, and 
$\Ft$ is a Cartan subalgebra of $\Fg$. 
Let $\bigl\{ \alpha_{j} \bigr\}_{j \in I}$ and 
$\bigl\{ \alpha_{j}^{\vee} \bigr\}_{j \in I}$ be the simple roots and the simple coroots for $\Fg$, 
respectively. Set $Q^{\vee,+}:=\sum_{j \in I} \BZ_{\ge 0} \alpha_{j}^{\vee}$. 
We denote by $\Lambda=\bigoplus_{i \in I} \BZ \vpi_{i}$ the integral weight lattice for $\Fg$, 
where $\vpi_{i}$ is the $i$-th fundamental weight for $i \in I$. 
Let us denote by $W = \langle s_{j} \mid j \in I \rangle \subset GL(\Ft^{\ast})$ 
the (finite) Weyl group of $\Fg$, 
where $s_{j}$ is the simple reflection in a simple root $\alpha_{j}$. 
In \cite{LNSX}, they gave a quantum $K$-theoretic divisor axiom 
for the flag manifold $X:=G/B$ as follows: in $K_T(\pt)=R(T) \cong 
\BZ[\be^{\nu} \mid \nu \in \Lambda]$, 
\begin{equation} \label{eq:thpi}
\langle \CO^{s_i}, \CO^w, \CO_x \rangle_{d} =  
 \langle \CO^w, \CO_x \rangle_{d} - 
 \sum_{ \bp \in \bR{w,x,d} }
(-1)^{\ell(\bp)}\be^{-\vpi_{i}+\wt(\eta_{\bp})}
\end{equation}
for $i \in I$ and $w,x \in W$, $d \in Q^{\vee,+}$, 
where $\CO_{u}$ and $\CO^{u}$ for $u \in W$ 
are the Schubert class and the opposite Schubert class 
in the equivariant $K$-theory ring $K_T(X)$ of 
the flag manifold $X:=G/B$, and 
$\langle \gamma_1, \gamma_2, \ldots ,\gamma_m \rangle_{d}$ 
for $\gamma_1, \gamma_2, \dots ,\gamma_m \in K_{T}(X)$ and $d \in Q^{\vee,+}$ 
denotes the corresponding $m$-point ($T$-equivariant) 
$K$-theoretic Gromov-Witten (KGW) invariant; see, e.g., \cite{LNSX}. 
In this paper, we focus only on the sum 
\begin{equation} \label{eq:corr}
 \sum_{ \bp \in \bR{w,x,d} }
(-1)^{\ell(\bp)}\be^{-\vpi_{i}+\wt(\eta_{\bp})}
\end{equation}
on the right-hand side of \eqref{eq:thpi} 
which is defined combinatorially in terms of 
the quantum Bruhat graph $\QBG(W)$ for $W$ (see Definition~\ref{idfn:QBG} below); 
we do not use 
the (geometric) definitions of $K_{T}(X)$, $\CO^{u}$, $\CO_{u}$, and the KGW invariants. 
The sum \eqref{eq:corr} is not cancellation-free in general. 
The purpose of this paper is to give (under some assumption) 
a cancellation-free version of the quantum $K$-theoretic divisor axiom for $X=G/B$ 
by removing all the cancellations from \eqref{eq:corr} 
in the case where $\vpi_{i}$ is quasi-minuscule in the sense that 
$\pair{\vpi_{i}}{\beta^{\vee}} \in \{0,1,2\}$ 
for all $\beta \in \Delta^{+}$, where $\Delta^{+}$ denotes the set of 
positive roots for $\Fg$, and $\beta^{\vee}$ denotes the coroot of $\beta$; 
for the list of quasi-minuscule fundamental weights, see \eqref{eq:list}. 
Remark that if $\Fg$ is of classical type (i.e., of type $A$, $B$, $C$, or $D$), 
then all the fundamental weights $\vpi_{i}$ are quasi-minuscule. 

Let us explain our formula more precisely. 
First, let us recall the definition of the quantum Bruhat graph $\QBG(W)$. 

\begin{idfn} \label{idfn:QBG}
The quantum Bruhat graph $\QBG(W)$ is the $\Delta^{+}$-labeled
directed graph whose vertices are the elements of $W$ and 
whose edges are of the following form: 
$x \edge{\alpha} y$, with $x, y \in W$ and $\alpha \in \Delta^{+}$, 
such that $y = x s_{\alpha}$ and either of the following holds: 
(B) $\ell(y) = \ell (x) + 1$; 
(Q) $\ell(y) = \ell (x) + 1 - 2 \pair{\rho}{\alpha^{\vee}}$, 
where $\rho:=\frac{1}{2}\sum_{\alpha \in \Delta^{+}} \alpha$. 
An edge satisfying (B) (resp., (Q)) is called a Bruhat edge (resp., quantum edge). 
\end{idfn}

Let $w,v \in W$, and let 
\begin{equation*}
w = u_{0} \edge{\beta_{1}} u_{1} \edge{\beta_{2}} \cdots \edge{\beta_{r}} u_{r} = v
\end{equation*}
be a shortest directed path from $w$ to $v$ in $\QBG(W)$. 
Then we set
\begin{equation*}
\qwt(w \Rightarrow v) :=  \sum_{ \begin{subarray}{c} 1 \le t \le r \\ 
\text{$u_{t-1} \edge{\beta_t} u_{t}$ is} \\
\text{a quantum edge} \end{subarray} } \beta_{t}^{\vee} \in Q^{\vee,+}; 
\end{equation*}
we know from \cite[Proposition~8.1]{LNSSS1} that 
$\qwt(w \Rightarrow v)$ does not depend on 
the choice of a shortest directed path $\bp$. 

Fix $i \in I$ such that $\vpi_{i}$ is quasi-minuscule. 
Let $\lhd$ be an arbitrary reflection order on $\Delta^{+}$ 
satisfying the condition that $\gamma \lhd \beta$ 
for all $\gamma \in \DJs$ and $\beta \in \DJp$, 
where $\J:=I \setminus \{i\}$, and 
$\DJs:=\Delta^{+} \cap \bigoplus_{j \in \J} \BZ \alpha_{j}$. 
Fix $w,x \in W$, and $d \in Q^{\vee,+}$. 
Here we define $\bR{w,x,d}$ to be 
the set of all directed paths 
$\bp: w = u_{0} \edge{\beta_{1}} u_{1} \edge{\beta_{2}} \cdots 
\edge{\beta_{r}} u_{r} =: \ed(\bp)$
starting at $w$ satisfying the conditions that 
\begin{equation}
\begin{cases}
r \ge 0, \quad 
\beta_{1} \lhd \cdots \lhd \beta_{r}, \quad 
\pair{\vpi_{i}}{\beta_{t}^{\vee}}=2 \ \text{for $1 \le t \le r$}, \\
\qwt (\ed(\bp) \Rightarrow x) \le d - \qwt(\bp), \quad 
\pair{\vpi_{i}}{d - \qwt(\bp)} = 0, \quad \ed(\bp) \in x\WJs, 
\end{cases}
\end{equation}
where we set $\WJs:=\langle s_{j} \mid j \in \J \rangle$, and 
for $\xi_{1},\xi_{2} \in \Ft$, we write $\xi_{1} \ge \xi_{2}$ 
if $\xi_{1} - \xi_{2} \in Q^{\vee,+}$. 

Further, we define the $w$-tilted Bruhat order 
$\tb{w}$ on $W$ as follows: $v_{1} \tb{w} v_{2}$ if 
there exists a shortest directed path in $\QBG(W)$ 
from $w$ to $v_{2}$ passing through $v_{1}$. 
We know from \cite[Theorem~7.1]{LNSSS1}
that the coset $x\WJs$ has a unique minimal element 
with respect to $\tb{w}${\rm;} we denote it by $\tbmin{x}{\J}{w}$. 

We now state our main formulas, whose proofs are completely combinatorial; 
here we set $Q^{+}:=\sum_{j \in I} \BZ_{\ge 0} \alpha_{j}$. 
%
%
\begin{ithm}[see Theorem~\ref{thm:main1}] \label{ithm:main1}
Let $w,x \in W$ and $d \in Q^{\vee,+}$ be such that 
$\# \bR{w,x,d} \in \{0,1\}$. We have
\begin{equation} \label{eq:main1x}
\langle \CO^{s_i}, \CO^{w}, \CO_{x} \rangle_{d} =  
\begin{cases}
1 - (-1)^{\ell(w) - \ell(x_{\min})} \be^{-\vpi_{i}+\wt(\eta)}
  & \text{\rm if $\# \bR{w,x,d} = 1$}, \\[1mm]
1 & \text{\rm if $\bR{w,x,d} = \emptyset$ and $\qwt(w \Rightarrow x) \le d$}, \\[1mm]
0 & \text{\rm if $\bR{w,x,d} = \emptyset$ and $\qwt(w \Rightarrow x) \not\le d$}, 
\end{cases}
\end{equation}
where $x_{\min}:=\tbmin{x}{\J}{w}$ and 
$\wt(\eta) := \frac{1}{2}w\vpi_{i}+\frac{1}{2}x\vpi_{i} \in \vpi_{i}-Q_{+}$. 
\end{ithm}
Here we make the following conjecture. 
%
%
\begin{iconj}[see Conjecture~\ref{conj:R1even} and also Remark~\ref{rem:conj}] 
\label{iconj:R1even}
Let $i \in I$ be such that $\vpi_{i}$ is quasi-minuscule. 
We have $\# \bR{w,x,d} \in \{ 1 \} \sqcup 2\BZ_{\ge 0}$ 
for all $w,x \in W$ and $d \in Q^{\vee,+}$. 
\end{iconj}
%
%
\begin{ithm}[see Theorem~\ref{thm:main3}] \label{ithm:main}
Assume that Conjecture~\ref{conj:R1even} is true for $\vpi_{i}$. 
Let $w,x \in W$, and $d \in Q^{\vee,+}$. If $\# \bR{w,x,d} \ge 2$, then 
$\langle \CO^{s_i}, \CO^{w}, \CO_{x} \rangle_{d} =  1$. 
\end{ithm}

We set
\begin{equation}
D := \sum_{\beta \in \Delta^{+}_{2}} \beta^{\vee} + 
\sum_{\gamma \in \DJs} \gamma^{\vee} \in Q^{\vee,+},
\end{equation}
where $\Delta^{+}_{2}:=\bigl\{ \beta \in \Delta^{+} \mid 
\pair{\vpi_{i}}{\beta^{\vee}} = 2 \bigr\}$, and write it as
$D=\sum_{j \in I} D_{j} \alpha_{j}^{\vee}$ 
with $D_{j} \in \BZ_{\ge 0}$ for $j \in I$. 

\begin{irem}[see Remark~\ref{rem:hR1even}]
In order to prove Conjecture~\ref{iconj:R1even}, 
it suffices to show that $\# \bR{w,x,d} \in \{ 1 \} \sqcup 2\BZ_{\ge 0}$ 
for all $w,x \in W$ and $d \in Q^{\vee,+}$ such that $d \le D$; 
notice that $W$ and $\bigl\{ d \in Q^{\vee,+} \mid d \le D \bigr\}$ are finite sets. 
\end{irem}

In the following theorem, we do not assume that 
Conjecture~\ref{conj:R1even} is true for $\vpi_{i}$. 

\begin{ithm}[see Corollary~\ref{cor:hR1even}] \label{ithm:hR1even}
Let $w,x \in W$, and let $d = \sum_{j \in I} d_{j}\alpha_{j}^{\vee} \in Q^{\vee,+}$ 
be such that $d_{j} \ge D_{j}$ for all $j \in \J=I \setminus \{i\}$. 
If $\# \bR{w,x,d} \ge 2$, then $\langle \CO^{s_i}, \CO^{w}, \CO_{x} \rangle_{d} =  1$. 
\end{ithm}

\begin{irem}[Positivity]
In Theorems~\ref{ithm:main1}, \ref{ithm:main}, and \ref{ithm:hR1even}, 
we see that 
\begin{equation*}
\langle \CO^{s_i}, \CO^{w}, \CO_{x} \rangle_{d} \in 
\BZ_{\ge 0}[\be^{-\beta}, \be^{\beta}-1 \mid \beta \in Q^{+}] 
\subset \BZ_{\ge 0}[\be^{\beta}-1 \mid \beta \in Q^{+} \cup (-Q^{+})]. 
\end{equation*}
In particular, 
$\langle \CO^{s_i}, \CO^{w}, \CO_{x} \rangle_{d}|_{\be \to 1} \in \BZ_{\ge 0}$ 
in the non-equivariant $K$-theory ring $K(X)$. 
\end{irem}

This paper is organized as follows. 
In Section~\ref{sec:prelim}, we recall the basic notation (for root systems, 
quantum Bruhat graphs, and KGW invariants) from \cite{LNSX}, 
and show some technical lemmas needed in the proofs of the theorems. 
In Subsection~\ref{subsec:mc}
give a cancellation-free quantum $K$-theoretic divisor axiom  
in the case where $\vpi_{i}$ is minuscule; see \eqref{eq:qkcf1a}.  
In Subsection~\ref{subsec:qmc}, 
we consider the case where $\vpi_{i}$ is quasi-minuscule , 
and prove Theorems~\ref{ithm:main1} and \ref{ithm:main} above. 
In Subsection~\ref{subsec:variation}, we introduce a variation $\hbR{w,x\WJs,k}$ 
of $\bR{w,x,d}$ for $w,x \in W$ and $k \in \BZ_{\ge 0}$, and prove that 
$\# \hbR{w,x\WJs,k} \in \{1\} \sqcup 2\BZ_{\ge 0}$ (cf. Conjecture~\ref{iconj:R1even}). 
As an application of this result, we prove Theorem~\ref{ithm:hR1even}. 
In Appendix~\ref{sec:example}, we give some examples of $\bR{w,x,d}$ 
in the case where $\Fg$ is of type $B_{n}$. 

\medskip

\paragraph{\bf Acknowledgments.}
D.S. would like to thank Cristian Lenart, Satoshi Naito, and Weihong Xu 
for related collaborations.
R.K. was partly supported by JST SPRING, Grant Number JPMJSP2124. 
D.S. was partly supported by JSPS Grant-in-Aid for Scientific Research (C) 23K03045.

%
\section{Preliminaries.}
\label{sec:prelim}
%
%
\subsection{Notation for root systems.} \label{subsec:alggrp}
In this paper, we will use the same notation as in \cite{LNSX}. 
Let $G$ be a connected, simply-connected, simple (linear) algebraic group over $\BC$, 
and $T$ a maximal torus of $G$. 
We set $\Fg := \Lie(G)$ and $\Ft := \Lie(T)$; 
$\Fg$ is a finite-dimensional simple Lie algebra over $\BC$, and 
$\Ft$ is a Cartan subalgebra of $\Fg$. 
We denote by $\pair{\cdot\,}{\cdot} : \Ft^{\ast} \times \Ft \rightarrow \BC$ 
the canonical pairing, where $\Ft^{\ast} := \Hom_{\BC}(\Ft, \BC)$. 
Let $\Delta \subset \Ft^{\ast}$ be the root system of $\Fg$, 
$\Delta^{+} \subset \Delta$ the set of positive roots, 
and $\{ \alpha_{j} \}_{j \in I} \subset \Delta^{+}$ the set of simple roots.
Denote by $\Delta^{-}:=-\Delta^{+}$ the set of negative roots. 
We denote by $\alpha^{\vee} \in \Ft$ the coroot of $\alpha \in \Delta$. 
Also, we denote by $\theta \in \Delta^{+}$ the highest root of $\Delta$, 
and set $\rho := (1/2) \sum_{\alpha \in \Delta^{+}} \alpha$. 
The root lattice $Q$ and the coroot lattice $Q^{\vee}$ of $\Fg$ are defined by 
$Q := \bigoplus_{j \in I} \BZ \alpha_{j}$ and 
$Q^{\vee} := \bigoplus_{j \in I} \BZ \alpha_{j}^{\vee}$, respectively. 
We set $Q^{\vee,+} := \sum_{j \in I} \BZ_{\ge 0} \alpha_{j}^{\vee}$. 
For $\xi,\zeta \in \Ft$, we write $\xi \ge \zeta$ if $\xi-\zeta \in Q^{\vee,+}$. 
For $i \in I$, the weight $\vpi_{i} \in \Ft^{\ast}$ 
given by $\pair{\vpi_{i}}{\alpha_{j}^{\vee}} = \delta_{i,j}$ for all $j \in I$, 
with $\delta_{i,j}$ the Kronecker delta, is called the $i$-th fundamental weight.
Denote by $\Lambda := \bigoplus_{j \in I} \BZ \vpi_{j}$
the (integral) weight lattice of $\Fg$, 
and by $\Lambda^{+} := \sum_{j \in I} \BZ_{\ge 0} \vpi_{j}$
the set of dominant (integral) weights. 
We denote by $\BZ[\Lambda]$ the group algebra of $\Lambda$, that is, 
the associative $\BZ$-algebra with a $\BZ$-basis $\bigl\{ \be^{\nu} \mid \nu \in \Lambda \bigr\}$, 
where the product is defined by $\be^{\mu} \be^{\nu} : = 
\be^{\mu + \nu}$ for $\mu,\,\nu \in \Lambda$. 

A reflection $s_{\alpha} \in GL(\Ft^{\ast})$, $\alpha \in \Delta$, 
is defined by $s_{\alpha} \mu := \mu - \pair{\mu}{\alpha^{\vee}} \alpha$ 
for $\mu \in \Ft^{\ast}$. Denote by $s_{j} := s_{\alpha_{j}}$ for $j \in I$ 
the simple reflection in $\alpha_{j}$. 
The (finite) Weyl group $W$ of $\Fg$ is defined to be the subgroup of $GL(\Ft^{\ast})$ 
generated by $\{ s_{j} \}_{j \in I}$, that is, $W : = \langle s_{j} \mid j \in I \rangle$. 
For $w \in W$, we denote by $\ell(w)$ the length of $w$. 

Let $\J$ be a subset of $I$. We set
\begin{align*}
& \QJ := \bigoplus_{j \in \J} \BZ \alpha_j, \qquad 
  \Delta^{\pm}_{\J} := \Delta^{\pm} \cap \QJ, \qquad
  \rho_{\J}:= (1/2) \sum_{\alpha \in \DJs} \alpha, \\
& \QJv := \bigoplus_{j \in \J} \BZ \alpha_{j}^{\vee}, \qquad
  \QJvp := \sum_{j \in \J} \BZ_{\ge 0} \alpha_{j}^{\vee}, \qquad 
  \WJs := \langle s_{j} \mid j \in \J \rangle. 
\end{align*}
For $w \in W$, let $\mcr{w}=\mcr{w}^{\J}$ and $\mxcr{w}=\mxcr{w}^{\J}$ 
denote the minimal(-length) and maximal(-length) coset representatives for the coset $w\WJs$, 
respectively. 
We set $\WJ := \bigl\{ \mcr{w}^{\J} \mid w \in W \bigr\} \subset W$. 

%
\subsection{The quantum Bruhat graph.}
\label{subsec:QBG}

\begin{dfn} \label{dfn:QBG} 
Let $\J$ be a subset of $I$. 
The (parabolic) quantum Bruhat graph on $\WJ$, 
denoted by $\QBG(\WJ)$, is the ($\DJp$)-labeled
directed graph whose vertices are the elements of $\WJ$ and 
whose edges are of the following form: 
$x \edge{\alpha} y$, with $x, y \in \WJ$ and $\alpha \in \DJp$, 
such that $y = \mcr{x s_{\alpha}}^{\J}$ and either of the following holds: 
(B) $\ell(y) = \ell (x) + 1$; 
(Q) $\ell(y) = \ell (x) + 1 - 2 \pair{\rho-\rho_{\J}}{\alpha^{\vee}}$.
An edge satisfying (B) (resp., (Q)) is called a Bruhat edge (resp., quantum edge). 
When $\J=\emptyset$ (note that $W^{\emptyset}=W$, $\rho_{\emptyset}=0$, 
and $\mcr{x}^{\emptyset}=x$ for all $x \in W$), 
we write $\QBG(W)$ for $\QBG(W^{\emptyset})$.
\end{dfn}


Let 
\begin{equation} \label{eq:dp}
\bp: w = u_{0} \edge{\beta_{1}} u_{1} \edge{\beta_{2}} \cdots \edge{\beta_{r}} u_{r} = v
\end{equation}
be a directed path in the quantum Bruhat graph $\QBG(W) = \QBG(W^{\emptyset})$; 
we simply write $\bp$ as $\bp: w \edge{\beta_1,\dots,\beta_r} v$. 
We set $\ed(\bp) := u_{r} = v$ and $\ell(\bp)=r$. 
A directed path $\bp$ is said to be trivial (resp., non-trivial) one 
if $\ell(\bp)=0$ (resp., $\ell(\bp) > 0$). 
When $\bp$ is non-trivial, 
we call $\IL{\bp} := \beta_{1}$ the initial label of $\bp$. 
For $\bp$ of the form \eqref{eq:dp}, we set
\begin{equation*}
\qwt(\bp) :=  \sum_{ \begin{subarray}{c} 1 \le t \le r \\ 
\text{$u_{t-1} \edge{\beta_t} u_{t}$ is} \\
\text{a quantum edge} \end{subarray} } \beta_{t}^{\vee} \in Q^{\vee,+}. 
\end{equation*}
Let $w,v \in W$, and let $\bp$ be a shortest directed path from $w$ to $v$ in $\QBG(W)$. 
We set $\qwt(w \Rightarrow v) : = \qwt(\bp)$; we know from \cite[Proposition~8.1]{LNSSS1} that 
$\qwt(w \Rightarrow v)$ does not depend on the choice of a shortest directed path $\bp$. 
%
%
\begin{prop}[{\cite[Proposition~8.1]{LNSSS1}}] \label{prop:qwt}
Let $w,v \in W$, and let $\bq$ be 
a directed path (not necessarily, shortest) 
from $w$ to $v$ in $\QBG(W)$. Then we have 
$\qwt(\bq) \ge \qwt(w \Rightarrow v)$. 
\end{prop}

We set $I_{\af}:=I \sqcup \{0\}$, and 
\begin{equation} \label{eq:zero}
\alpha_{0}:=-\theta \in \Delta^{-}, \qquad 
s_{0}:=s_{\theta} \in W. 
\end{equation}
%
%
\begin{lem} \label{lem:edge}
Let $w \in W$, and $j \in I_{\af}$. 
If $w^{-1}\alpha_{j} \in \Delta^{+}$, then 
there exists an edge $w \edge{w^{-1}\alpha_{j}} s_{j}w$ in $\QBG(W)$. 
This edge is a Bruhat edge (resp., a quantum edge) if $j \ne 0$ (resp., $j = 0$).
\end{lem}
%
%
\begin{lem} \label{lem:qL}
Let $\J$ be a subset of $I$. For $w \in W$, there exists a sequence 
$j_{1},j_{2},\dots,j_{m}$ of elements in $I_{\af}=I \sqcup \{0\}$ such that 
\begin{equation} \label{eq:qL}
(s_{j_{k-1}} \cdots s_{j_{1}}w)^{-1}\alpha_{j_{k}} \in \DJp
\qquad \text{\rm and} \qquad
s_{j_{m}} \cdots s_{j_{1}} w \in \WJs; 
\end{equation}
in this case, if we set $w_{k}:=s_{j_{k}} \cdots s_{j_{1}}w$ for $0 \le k \le m$, then
\begin{equation}
w = w_{0} \edge{ w_{0}^{-1}\alpha_{j_1} } 
    w_{1} \edge{ w_{1}^{-1}\alpha_{j_2} } \cdots \edge{ w_{m-1}^{-1}\alpha_{j_m} } w_{m} \in \WJs. 
\end{equation}
\end{lem}

\begin{proof}
It follows from \cite[Lemma 6.12 and Definition~6.14]{LNSSS1} that there exist 
a sequence $\mcr{w} = x_{0},x_{1},\dots,x_{m-1},x_{m} = e$ of elements in $\WJ$ 
and a sequence $j_{1},j_{2},\dots,j_{m}$ of elements in $I_{\af}$
such that
\begin{equation}
\mcr{w} = x_{0} \edge{ x_{0}^{-1}\alpha_{j_1} } x_{1} 
\edge{ x_{1}^{-1}\alpha_{j_2} } \cdots \edge{ x_{m-1}^{-1}\alpha_{j_m} } x_{m} = e
\end{equation}
in the parabolic quantum Bruhat graph $\QBG(\WJ)$; note that 
$x_{k-1}^{-1}\alpha_{j_k} \in \DJp$ for all $1 \le k \le m$. 
By \cite[Lemma 8.1]{LNSSS2}, we deduce that there exists a directed path 
\begin{equation*}
w = w_{0} \edge{ w_{0}^{-1}\alpha_{j_1} } w_{1} 
\edge{ w_{1}^{-1}\alpha_{j_2} } \cdots 
\edge{ w_{m-1}^{-1}\alpha_{j_m} } w_{m}
\end{equation*}
with $w_{k-1}^{-1}\alpha_{j_k} \in \DJp$ and $\mcr{w_{k}}=x_{k}$ 
for all $1 \le k \le m$. Thus we have proved the lemma. 
\end{proof}
%
%
\begin{lem}[{\cite[Lemma 7.7]{LNSSS1}}] \label{lem:DL}
Let $w,v \in W$, and 
\begin{equation*}
w = u_{0} \edge{\beta_1} u_{1} \edge{\beta_2} \cdots \edge{\beta_r} u_r = v
\end{equation*}
a directed path from $w$ to $v$ in $\QBG(W)$. Let $j \in I_{\af}=I \sqcup \{0\}$. 
\begin{enu}
\item Assume that $w^{-1}\alpha_{j} \in \Delta^{+}$ and $v^{-1}\alpha_{j} \in \Delta^{-}$. 
Let $1 \le t \le r$ be such that $u_{p}^{-1}\alpha_{j} \in \Delta^{+}$ for all $0 \le p \le t-1$ 
and $u_{t}^{-1}\alpha_{j} \in \Delta^{-}$. Then there exists a directed path $\bq$ 
from $s_{j}w$ to $v$ of the form
\begin{equation*}
\bq : s_{j}w = s_{j}u_{0} \edge{\beta_1} \cdots \edge{\beta_{t-1}} s_{j}u_{t-1} = u_{t}
\edge{\beta_{t+1}} u_{t+1} \edge{\beta_{t+2}} \cdots \edge{\beta_{r}} u_{r} = v, 
\end{equation*}
where $\beta_{t} = u_{t-1}^{-1}\alpha_{j}$. 
We have $\ell(\bq) = \ell(\bp)-1$ and $\qwt(\bq) = 
\qwt(\bp) - \delta_{j0}w^{-1}\alpha_{j}^{\vee}$. 
If $\bp$ is a shortest directed path from $w$ to $v$, then 
$\bq$ is a shortest directed path from $s_{j}w$ to $v$. 

\item Assume that $w^{-1}\alpha_{j} \in \Delta^{+}$ and $v^{-1}\alpha_{j} \in \Delta^{-}$. 
Let $1 \le t \le r$ be such that $u_{p}^{-1}\alpha_{j} \in \Delta^{-}$ for all $t \le p \le r$ 
and $u_{t-1}^{-1}\alpha_{j} \in \Delta^{+}$. Then there exists a directed path $\bq$ 
from $w$ to $s_{j}v$ of the form
\begin{equation*}
\bq : w = u_{0} \edge{\beta_1} \cdots \edge{\beta_{t-1}} u_{t-1} = s_{j}u_{t}
\edge{\beta_{t+1}} s_{j}u_{t+1} \edge{\beta_{t+2}} \cdots \edge{\beta_{r}} s_{j}u_{r} = s_{j}v, 
\end{equation*}
with $\beta_{t} = u_{t-1}^{-1}\alpha_{j}$. 
We have $\ell(\bq) = \ell(\bp)-1$ and $\qwt(\bq) = \qwt(\bp) + \delta_{j0}v^{-1}\alpha_{j}^{\vee}$. 
If $\bp$ is a shortest directed path from $w$ to $v$, then 
$\bq$ is a shortest directed path from $w$ to $s_{j}v$.

\item If $w^{-1}\alpha_{j} \in \Delta^{+}$ and $v^{-1}\alpha_{j} \in \Delta^{+}$ or 
if $w^{-1}\alpha_{j} \in \Delta^{-}$ and $v^{-1}\alpha_{j} \in \Delta^{-}$, then 
there exists a directed path $\bq$ from $s_{j}w$ to $s_{j}v$ such that 
$\ell(\bq) = \ell(\bp)$ and 
$\qwt(\bq) = \qwt(\bp) -\delta_{j0}w^{-1}\alpha_{j}^{\vee} + \delta_{j0}v^{-1}\alpha_{j}^{\vee}$. 
If $\bp$ is a shortest directed path from $w$ to $v$, then 
$\bq$ is a shortest directed path from $s_{j}w$ to $s_{j}v$. 
\end{enu}

\end{lem}

\begin{proof}
The existence of $\bq$ follows from 
\cite[Lemma 7.7 and its proof]{LNSSS1}. 
Let us show the shortestness of $\bq$. 
We give a proof only for part (1); 
the proofs for the other parts are similar. 
Suppose, for a contradiction, that $\bq$ is not a shortest 
directed path from $s_{j}w$ to $v$. 
Since $\ell(s_{j}w \Rightarrow v) \equiv \ell(\bq)$ mod $2$, 
it follows that $\ell(s_{j}w \Rightarrow v) \le (r-1) - 2 = r-3$. By Lemma~\ref{lem:edge}, 
we have an edge $w \edge{w^{-1}\alpha_{j}} s_{j}w$. 
Concatenating this edge and a shortest directed path from $s_{j}w$ to $v$, 
we obtain a directed path from $w$ to $v$ of length less than or equal to $r-2$. 
Hence, $\ell(w \Rightarrow v) \le r-2 < \ell(\bp)$, 
which contradicts the shortestness of $\bp$. 
\end{proof}
%
%
\begin{dfn}[tilted Bruhat order] \label{dfn:tilted}
For each $w \in W$, we define the $w$-tilted Bruhat order $\tb{w}$ on $W$ as follows:
for $v_{1},v_{2} \in W$, 
%
%
\begin{equation} \label{eq:tilted}
v_{1} \tb{w} v_{2} \iff \ell(w \Rightarrow v_{2}) = 
 \ell(w \Rightarrow v_{1}) + \ell(v_{1} \Rightarrow v_{2}).
\end{equation}
Namely, $v_{1} \tb{w} v_{2}$ if and only if 
there exists a shortest directed path in $\QBG(W)$ 
from $w$ to $v_{2}$ passing through $v_{1}$; 
or equivalently, if and only if 
the concatenation of a shortest directed path 
from $w$ to $v_{1}$ and one from $v_{1}$ to $v_{2}$ 
is one from $w$ to $v_{2}$. 
\end{dfn}
%
%
\begin{prop}[{\cite[Theorem~7.1]{LNSSS1}}] \label{prop:tbmin}
Let $\J$ be a subset of $I$, and let $w \in W$. 
Then each coset $x\WJs$ for $x \in W$ has a unique minimal element
with respect to $\tb{w}${\rm;} we denote it by $\tbmin{x}{\J}{w}$. 
\end{prop}

Let $\lhd$ be a reflection (convex) order on $\Delta^{+}$; 
see, e.g., \cite[Section 2.2]{KNS}.
A directed path $\bp: w \edge{\beta_1,\dots,\beta_r} v$ of the form \eqref{eq:dp} 
is said to be label-increasing with respect to $\lhd$ if $\beta_{1} \lhd \cdots \lhd \beta_{r}$. 
%
%
\begin{thm}[{see, e.g., \cite[Theorem~7.4]{LNSSS1}}] \label{thm:inc}
For each $w,v \in W$, 
there exists a unique label-increasing directed path from $w$ to $v$ 
in the quantum Bruhat graph $\QBG(W)$. Moreover, it is a shortest directed path 
from $w$ to $v$, and lexicographically-minimal among the shortest directed paths 
from $w$ to $v$ in the following sense: 
If $\bp:w \edge{\beta_1,\dots,\beta_r} v$ is the label-increasing 
directed path from $w$ to $v$ (note that $r=\ell(w \Rightarrow v)$) and 
$\bq : w \edge{\gamma_1,\dots,\gamma_r} v$ is a shortest directed path 
from $w$ to $v$ with $\bq \ne \bp$, 
then there exists $1 \le t \le r$ such that 
$\beta_{p} = \gamma_{p}$ for all $1 \le p \le t-1$ and 
$\beta_{t} \lhd \gamma_{t}$. 
\end{thm}

Let $\J$ be a subset of $I$. 
As in \cite[(2.4)]{KNS}, let $\lhd$ be an arbitrary reflection 
order on $\Delta^{+}$ satisfying the condition that 
%
%
\begin{equation} \label{eq:ro}
\gamma \lhd \beta \quad 
\text{for all $\gamma \in \DJs$ and $\beta \in \DJp$}. 
\end{equation}
For $w \in W$, let $\bQBG{w}$ denote 
the set of all label-increasing directed paths $\bp$ 
in $\QBG(W)$ starting at $w$, and 
satisfying the condition that 
all the labels of the edges in $\bp$ are contained in $\DJp$: 
%
%
\begin{equation} \label{eq:bQBG}
\bp: 
\underbrace{w = z_{0} \edge{\beta_{1}} z_{1} \edge{\beta_{2}} \cdots \edge{\beta_{r}} z_{r},}_{
\text{directed path in $\QBG(W)$}} \quad \text{where} \quad 
\begin{cases}
r \ge 0, \\[1mm]
\text{$\beta_{t} \in \DJp$, $1 \le t \le r$}, \\[1mm]
\beta_{1} \lhd \beta_{2} \lhd \cdots \lhd \beta_{r}; 
\end{cases}
\end{equation}
note that $\bp$ is a shortest directed path from $w$ to $z_{r}=\ed(\bp)$. 
Let $\bt_{w}$ denote the trivial directed path (of length $0$) 
starting at $w$ and ending at $w$; 
note that $\bt_{w} \in \bQBG{w}$. 

\begin{lem} \label{lem:WJs}
Let $\J$ be a subset of $I$, and let $x \in W$. 
If $v_{1},v_{2} \in x\WJs$, then the labels of 
a shortest directed path from $v_{1}$ to $v_{2}$ are contained in $\DJs$. 
In particular, $\qwt(v_{1} \Rightarrow v_{2}) \in \QJvp$. 
\end{lem}

\begin{proof}
In this proof, we fix a reflection order $\prec$ satisfying 
the condition that $\beta \prec \gamma$ 
for $\beta \in \DJp$ and $\gamma \in \DJs$. 
We first show the following claim. 

\begin{claim} \label{c:WJs}
Let $x_{1},x_{2} \in x\WJs$. The labels of 
the label-increasing directed path (with respect to $\prec$) 
from $x_{1}$ to $x_{2}$ are all contained in $\DJs$.
\end{claim}

\noindent
{\it Proof of Claim~\ref{c:WJs}.}
Write $x_{1}=\mcr{x}y_{1}$ and $x_{2}=\mcr{x}y_{2}$ 
with $y_{1},y_{2} \in \WJs$, respectively, 
where $\mcr{x}=\mcr{x}^{\J}$ is the minimal coset representative for $x \WJs$. 
Here we remark that $\WJs$ is the Weyl group for 
the root system $\Delta_{\J}$, and the restriction of $\prec$ to 
$\DJs$ gives a reflection order on the positive root system $\DJs$. Hence 
there exists a label-increasing directed path 
\begin{equation*}
y_{1} = z_{0} \edge{\gamma_{1}} \cdots \edge{\gamma_{r}} z_{r} = y_{2}
\end{equation*}
from $y_{1}$ to $y_{2}$ in the quantum Bruhat graph $\QBG(\WJs)$ for $\WJs$, 
where $\gamma_{1},\dots,\gamma_{r} \in \DJs$ and $\gamma_{1} \prec \cdots \prec \gamma_{r}$. 
Because $\ell(\mcr{x}z) = \ell(\mcr{x}) + \ell(z)$ for all $z \in \WJs$, we deduce that 
\begin{equation*}
x_{1}=\mcr{x}y_{1} = \mcr{x}z_{0} \edge{\gamma_{1}} 
\cdots \edge{\gamma_{r}} \mcr{x}z_{r} = \mcr{x}y_{2} = x_{2}
\end{equation*}
is a directed path in $\QBG(W)$; notice that this is a label-increasing directed path 
from $v_{1}$ to $v_{2}$. Thus we have shown the claim. \bqed

\medskip

Now, let $v_{1},v_{2} \in x\WJs$, and set $r:=\ell(v_{1} \Rightarrow v_{2})$. 
We show by induction on $r$ that if 
\begin{equation*}
v_{1}=u_{0} \edge{\phi_{1}} \cdots \edge{\phi_{r}} u_{r} = v_{2}
\end{equation*}
is an arbitrary shortest directed path from $v_{1}$ to $v_{2}$ in $\QBG(W)$, 
then $\phi_{t} \in \DJs$ for all $1 \le t \le r$. 
If $r=0$, then the assertion is obvious. Assume that $r > 0$. 
By Claim~\ref{c:WJs}, the initial label $\gamma$ of 
the label-increasing directed path 
from $x_{1}$ to $x_{2}$ is contained in $\DJs$. 
Thus, by the lexicographic-minimality of 
label-increasing directed paths (see Theorem~\ref{thm:inc}), 
we have $\gamma \preceq \phi_{1}$, which implies that $\phi_{1} \in \DJs$ 
and $u_{1}=u_{0}s_{\phi_1} = v_{1}s_{\phi_{1}} \in x\WJs$. 
By applying the induction hypothesis to 
the shortest directed path 
$u_{1} \edge{\phi_{2}} \cdots \edge{\phi_{r}} u_{r} = v_{2}$ 
from $u_{1}$ to $v_{2}$, 
we get $\phi_{t} \in \DJs$ for all $2 \le t \le r$. 
Thus we have proved the lemma. 
\end{proof}

%
\subsection{Quantum Lakshmibai-Seshadri paths.}
\label{subsec:QLS}

Let $\lambda \in \Lambda^{+}$ be a dominant (integral) weight, 
and take 
%
%
\begin{equation} \label{eq:J}
\J=\J_{\lambda}:= 
\bigl\{ j \in I \mid \pair{\lambda}{\alpha_{j}^{\vee}}=0 \bigr\}.
\end{equation}
%
%
\begin{dfn} \label{dfn:QBa}
For a rational number $0 \le a < 1$, 
we define $\QBG_{a\lambda}(\WJ)$ (resp., $\QBG_{a\lambda}(W)$) 
to be the subgraph of $\QBG(\WJ)$ (resp., $\QBG(W)$)
with the same vertex set but having only those directed edges of 
the form $x \edge{\alpha} y$ for which 
$a\pair{\lambda}{\alpha^{\vee}} \in \BZ$ holds. 
\end{dfn}
%
%
\begin{dfn}[{\cite[Section~3.2]{LNSSS2}}] \label{dfn:QLS}
A quantum  Lakshmibai-Seshadri path (QLS path for short) 
of shape $\lambda$ is a pair 
%
%
\begin{equation} \label{eq:QLS}
\eta = (\bv \,;\, \ba) = 
(v_{1},\,\dots,\,v_{s} \,;\, a_{0},\,a_{1},\,\dots,\,a_{s}), \quad s \ge 1, 
\end{equation}
consisting of a sequence $v_{1},\,\dots,\,v_{s}$ 
of elements in $\WJ$, with $v_{k} \ne v_{k+1}$ 
for any $1 \le k \le s-1$, and an increasing sequence 
$0 = a_0 < a_1 < \cdots  < a_s =1$ of rational numbers 
satisfying the condition that there exists a directed path 
in $\QBG_{a_{k}\lambda}(\WJ)$ from $v_{k+1}$ to  $v_{k}$ 
for each $k = 1,\,2,\,\dots,\,s-1$. 
\end{dfn}

Let $\QLS(\lambda)$ denote the set of all QLS paths of shape $\lambda$. 
For $\eta \in \QLS(\lambda)$ of the form \eqref{eq:QLS}, we set
\begin{equation} \label{eq:wt}
\wt (\eta) := \sum_{k=1}^{s} (a_{k}-a_{k-1}) v_{k}\lambda \in \Lambda. 
\end{equation}
Let $i \in I$, and consider the case of $\lambda=\vpi_{i}$; 
note that $\J=\J_{\vpi_{i}}$ is identical to $I \setminus \{i\}$ in this case.  
We fix $N = N_{i} \in \BZ_{\ge 1}$ satisfying the following condition: 
\begin{equation} \label{eq:cN}
N/\pair{\vpi_{i}}{\alpha^{\vee}} \in \BZ \quad 
\text{for all $\alpha \in \Delta^{+}$ such that $\pair{\vpi_{i}}{\alpha^{\vee}} \ne 0$}.
\end{equation}
By the definition of QLS paths of shape $\vpi_{i}$, 
we see that if 
\begin{equation} \label{eq:QLS1}
\eta = (v_{1},\,\dots,\,v_{s} \,;\, 
a_{0},\,a_{1},\,\dots,\,a_{s}) \in \QLS(\vpi_{i}),
\end{equation}
then $Na_{k} \in \BZ$ for all $0 \le k \le s$; 
we write $\eta$ as
\begin{equation} \label{eq:QLS2}
\eta = ( \underbrace{ v_{1},\dots,v_{1} }_{N(a_{1}-a_{0}) \text{ times}}, 
\underbrace{ v_{2},\dots,v_{2} }_{N(a_{2}-a_{1}) \text{ times}},\,\dots,\,
\underbrace{ v_{s},\dots,v_{s} }_{N(a_{s}-a_{s-1}) \text{ times}}). 
\end{equation}

Now, let $i \in I$, and set $\J=\J_{\vpi_{i}}=I \setminus \{i\}$ as above. 
Fix a reflection order $\lhd$ such that $\gamma \lhd \beta$ 
for all $\gamma \in \DJs$ and $\beta \in \DJp$; see \eqref{eq:ro}. 
For $w \in W$, we set 
\begin{equation}
\bQLS{w}: = 
\left\{ \bp=(\bp_{N},\dots,\bp_{2},\bp_{1}) \ \Biggm| \ 
\begin{array}{l}
\text{for all $1 \le k \le N$, $\bp_{k} \in \bQBG{\ed(\bp_{k+1})}$, and}  \\[2mm]
\text{$\bp_{k}$ is a directed path in $\QBG_{((k-1)/N)\vpi_{i}}(W)$}
\end{array} \right\}, 
\end{equation}
where $\bp_{N+1}$ is considered to be the trivial directed path $\bt_{w}$, 
and hence $\ed(\bp_{N+1})=w$. We know from \cite[Section~2.3]{LNSX} that 
for $\bp=(\bp_{N},\dots,\bp_{2},\bp_{1}) \in \bQLS{w}$, 
%
%
\begin{equation} \label{eq:etap}
\eta_{\bp} := 
(\mcr{\ed(\bp_{2})}^{\J},\dots,\mcr{\ed(\bp_{N})}^{\J},\mcr{\ed(\bp_{N+1})}^{\J}=\mcr{w}^{\J}) \in \QLS(\vpi_{i}). 
\end{equation}
For $\bp=(\bp_{N},\dots,\bp_{2},\bp_{1}) \in \bQLS{w}$, we set
\begin{equation}
\begin{split}
& \ell(\bp) := \sum_{k = 1}^{N} \ell(\bp_{k}), \qquad
  \ed(\bp) := \ed(\bp_{1}), \\
& \qwt(\bp) := \sum_{k = 1}^{N} \qwt(\bp_{k}), \qquad
  \qwt_{2}(\bp) := \sum_{k = 2}^{N} \qwt(\bp_{k}) = \qwt(\bp)-\qwt(\bp_{1}). 
\end{split}
\end{equation}

%
\subsection{$K$-theoretic Gromov-Witten invariants.}
\label{subsec:KGW}

Fix a Borel subgroup $B$ such that 
$T \subset B \subset G$.
The opposite Borel subgroup $B^{-} \subset G$ is 
the unique Borel subgroup (containing $T$) such that $B\cap B^{-} = T$. 
The Weyl group $W$ of $G$ can be identified with $N_G(T)/T$, 
where $N_G(T)$ is the normalizer of $T$ in $G$. 
Let $X=G/B$ be the flag manifold. 
Any Weyl group element $w \in W$ defines 
the Schubert variety $X_w = \ol{BwB/B}$ and 
the opposite Schubert variety
$X^w=\ol{B^{-}wB/B}$ in $X$; 
note that $\dim X_w = \codim X^w = \ell(w)$.
We denote by $K_T(X)$ 
the Grothendieck group of 
$T$-equivariant algebraic vector bundles on $X = G/B$. 
This ring is an algebra over $K_T(\pt)=R(T)$, 
the representation ring of $T$, which is identified with 
the group algebra $\BZ[\Lambda]$ of $\Lambda$. 
The equivariant $K$-theory ring $K_T(X)$ of 
the flag manifold $X=G/B$ has two $K_T(\pt)$-bases 
$\bigl\{ \CO_{w} \mid w \in W \bigr\}$ and 
$\bigl\{ \CO^{w} \mid w \in W \bigr\}$, 
where $\CO_{w}=[\CO_{X_w}]$ and $\CO^{w}=[\CO_{X^w}]$ are 
the Schubert class and the opposite Schubert class
defined by the structure sheaves of 
the Schubert variety $X_w$ and 
the opposite Schubert variety $X^w$ for $w \in W$, respectively. 
For classes $\gamma_{k} \in K_{T}(X)$, $1 \leq k \leq m$, and $d \in Q^{\vee,+}$, 
we denote by 
\begin{equation*}
\langle \gamma_1, \gamma_2, \ldots ,\gamma_m \rangle_{d} 
\in K_{T}(\pt) = R(T)
\end{equation*}
the corresponding $m$-point ($T$-equivariant) 
$K$-theoretic Gromov-Witten (KGW) invariant; see, e.g., \cite{LNSX} 
(in this paper, we do not use the definition of the KGW invariants). 
%
%
\begin{prop}[{\cite[Lemma~4.1]{LNSX}}] \label{prop:twp}
Let $d \in Q^{\vee,+}$, and $w,x \in W$. 
Then, we have 
%
%
\begin{equation} \label{eq:2pt-qwt}
\langle \CO^{w}, \CO_{x} \rangle_{d} = 
\begin{cases}
1 & \text{\rm if $\qwt(w \Rightarrow x) \le d$}, \\
0 & \text{\rm otherwise}.
\end{cases}
\end{equation}
\end{prop}

Let $i \in I$, and set $\J=\J_{\vpi_{i}}=I \setminus \{i\}$ as above. 
Also, recall that $\lhd$ is a reflection order such that 
$\gamma \lhd \beta$ for all $\gamma \in \DJs$ and $\beta \in \DJp$, 
and that $N=N_{i}$ satisfies condition \eqref{eq:cN}. 
Let $w,x \in W$, and $d \in Q^{\vee,+}$. We set
%
%
\begin{equation} \label{eq:bQLSwxd}
\bQLS{w,x,d} := 
\bigl\{ \bp=(\bp_{N},\dots,\bp_{2},\bp_{1}) \in \bQLS{w} \mid 
\qwt (\ed(\bp) \Rightarrow x) \le d - \qwt(\bp) \bigr\}. 
\end{equation}
Also, we define
%
%
\begin{equation} \label{eq:pbR}
\bR{w,x,d} := 
\left\{ \bp=(\bp_{N},\dots,\bp_{2},\bp_{1}) \in \bQLS{w,x,d} \ \Biggm| \ 
\begin{array}{c}
\pair{\vpi_{i}}{d - \qwt_{2}(\bp)} = 0 \\[1.5mm]
\ell(\bp_{1})=0, \, \ed(\bp) \in x\WJs
\end{array}
\right\}. 
\end{equation}
In Appendix~\ref{sec:example}, we give some examples of $\bR{w,x,d}$ 
in the case where $\Fg$ is of type $B_{n}$. 
%
%
\begin{thm}[{\cite[Theorem 3.2]{LNSX}}] \label{thm:thp}
Let $i \in I$, and let $w,x \in W$, $d \in Q^{\vee,+}$. 
Then we have 
\begin{equation} \label{eq:thp1}
\langle \CO^{s_i}, \CO^w, \CO_x \rangle_{d} =  
 \langle \CO^w, \CO_x \rangle_{d} - 
 \sum_{ \bp \in \bR{w,x,d} }
(-1)^{\ell(\bp)}\be^{-\vpi_{i}+\wt(\eta_{\bp})}.
\end{equation}
\end{thm}

The sum $\sum_{ \bp \in \bR{w,x,d} }
(-1)^{\ell(\bp)}\be^{-\vpi_{i}+\wt(\eta_{\bp})}$ 
on the right-hand side of the formula above is not cancellation-free in general. 
The purpose of this paper is to remove the cancellations from this sum 
in the case where $N=N_{i} \le 2$. 

\section{Main theorem.}
\label{sec:main}

Let $w,x \in W$, and $d=\sum_{j \in I} d_{j}\alpha_{j}^{\vee} \in Q^{\vee,+}$. 
Let $i \in I$, and set $\J=\J_{\vpi_{i}}=I \setminus \{i\}$ as in the previous section. 
Recall that $\lhd$ is a reflection order such that 
$\gamma \lhd \beta$ for all $\gamma \in \DJs$ and $\beta \in \DJp$, 
and that $N=N_{i}$ satisfies condition \eqref{eq:cN}. 

\subsection{Minuscule case.}
\label{subsec:mc}

Assume that $\vpi_{i}$ is minuscule, i.e., 
$\pair{\vpi_{i}}{\beta^{\vee}} \in \{0,1\}$ for all $\beta \in \Delta^{+}$; 
for the list of minuscule fundamental weights, 
see \eqref{eq:list} below. We can take $N=N_{i}:=1$ in this case. 
We deduce that if $\bR{w,x,d} \ne \emptyset$, then
\begin{equation} \label{eq:min}
x\WJs = w\WJs, \qquad d_{i}=0, \qquad 
\text{and} \qquad \qwt(w \Rightarrow x) \le d; 
\end{equation}
in this case, $\bR{w,x,d} = \bigl\{ \bt_{w} \bigr\}$. 
Thus we obtain
\begin{align}
\langle \CO^{s_i}, \CO^{w}, \CO_{x} \rangle_{d} & =  
 \langle \CO^{w}, \CO_{x} \rangle_{d} - 
 \sum_{ \bp \in \bR{w,x,d} }
 (-1)^{\ell(\bp)}\be^{-\vpi_{i}+\wt(\eta_{\bp})} \nonumber \\[3mm]
& = 
 \begin{cases}
 1 - \be^{-\vpi_{i}+w\vpi_{i}} & \text{if \eqref{eq:min} holds}, \\
 1 & \text{if \eqref{eq:min} does not hold, and $\qwt(w \Rightarrow x) \le d$}, \\
 0 & \text{if \eqref{eq:min} does not hold, and $\qwt(w \Rightarrow x) \not\le d$}.
 \end{cases} \label{eq:qkcf1a}
\end{align}

\subsection{Quasi-minuscule case.}
\label{subsec:qmc}

Fix $i \in I$ such that $\vpi_{i}$ is quasi-minuscule in the sense that 
$\pair{\vpi_{i}}{\beta^{\vee}} \in \{0,1,2\}$ for all $\beta \in \Delta^{+}$. 
The following table lists the indices $i$ for which 
$\vpi_{i}$ is minuscule or quasi-minuscule; we use 
the numbering of the Dynkin diagrams from \cite[Section 11.4]{H}:
%
%
\begin{equation} \label{eq:list}
\begin{array}{c||c|c|c}
 & \text{minuscule} & \text{quasi-minuscule} & \\ \hline\hline
A_{n} & \text{all $i \in I$} & \text{all $i \in I$} \\ \hline
B_{n} & n & \text{all $i \in I$} & \text{$\alpha_{n}$ is the unique short simple root} \\ \hline
C_{n} & 1 & \text{all $i \in I$} & \text{$\alpha_{n}$ is the unique long simple root} \\ \hline
D_{n} & 1, n-1, n & \text{all $i \in I$} & \\ \hline
E_{6} & 1,6 & 1,2,3,5,6 & \\ \hline
E_{7} & 7 & 1,2,6,7 & \\ \hline
E_{8} & \text{none} & 1,8 & \\ \hline
F_{4} & \text{none} & 1,4 & \text{$\alpha_{1}$ and $\alpha_{2}$ are long simple roots} \\ \hline
G_{2} & \text{none} & 1 & \text{$\alpha_{2}$ is a long simple root}
\end{array}
\end{equation}
Remark that if $\Fg$ is of classical type, then all the fundamental weights $\vpi_{i}$ 
are quasi-minuscule. We can take $N=N_{i}:=2$ in this case. 
It follows from the definitions that 
\begin{equation*}
\bQLS{w}: = 
\left\{ \bp=(\bp_{2},\bp_{1}) \ \Biggm| \ 
\begin{array}{l}
\text{$\bp_{1} \in \bQBG{\ed(\bp_{2})}$, $\bp_{2} \in \bQBG{w}$, and}  \\[2mm]
\text{$\bp_{2}$ is a directed path in $\QBG_{(1/2)\vpi_{i}}(W)$}
\end{array} \right\}; 
\end{equation*}
\begin{equation*}
\bQLS{w,x,d} := 
\bigl\{ \bp=(\bp_{2},\bp_{1}) \in \bQLS{w} \mid 
\qwt (\ed(\bp) \Rightarrow x) \le d - \qwt(\bp) \bigr\}; 
\end{equation*}
\begin{equation*}
\bR{w,x,d} := 
\left\{ \bp=(\bp_{2},\bp_{1}) \in \bQLS{w,x,d} \ \Biggm| \ 
\begin{array}{c}
\pair{\vpi_{i}}{d - \qwt_{2}(\bp)} = 0 \\[1.5mm]
\ell(\bp_{1})=0, \, \ed(\bp) \in x\WJs
\end{array}
\right\}; 
\end{equation*}
note that if $\bp=(\bp_{2},\bp_{1}) \in \bR{w,x,d}$, then 
$\qwt_{2}(\bp) = \qwt(\bp_{2})$ and $\ed(\bp) = \ed(\bp_{2})$. 
Also we have $\eta_{\bp} = (\mcr{x}^{\J},\mcr{w}^{\J}) \in \QLS(\vpi_{i})$ 
for all $\bp=(\bp_{2},\bp_{1}) \in \bR{w,x,d}$; we set
$\eta = (\mcr{x}^{\J},\mcr{w}^{\J}) \in \QLS(\vpi_{i})$. 
We can rewrite \eqref{eq:thp1} as
%
%
\begin{equation} \label{eq:thp-qm}
\langle \CO^{s_i}, \CO^{w}, \CO_{x} \rangle_{d} =  
 \langle \CO^{w}, \CO_{x} \rangle_{d} - \be^{-\vpi_{i}+\wt(\eta)}
 \sum_{ \bp \in \bR{w,x,d} }
 (-1)^{\ell(\bp)}. 
\end{equation}

\begin{rem}
Let $\bp,\bp' \in \bR{w,x,d}$. 
By the uniqueness of label-increasing directed paths, 
$\ed(\bp)=\ed(\bp')$ if and only if $\bp=\bp'$. 
\end{rem}
%
%
\begin{prop} \label{prop:Rtbmin}
Keep the notation and setting above. 
Assume that $\bR{w,x,d} \ne \emptyset$. 
Let $v \in x\WJs$ be such that $v = \ed(\bp) = \ed(\bp_{2})$ for some (unique) 
$\bp=(\bp_{2},\bp_{1}) \in \bR{w,x,d}$, and let $u \in x\WJs$ be such that $u \tb{w} v$ 
with respect to the $w$-tilted Bruhat order $\tb{w}$. 
Then there exists a unique $\bq = (\bq_{2},\bq_{1}) \in \bR{w,x,d}$ such that 
$u = \ed(\bq) = \ed(\bq_{2})$. Therefore there exists a unique $\bp_{\min} \in \bR{w,x,d}$ 
such that 
\begin{equation}
\ed(\bp_{\min}) = \tbmin{x}{\J}{w}. 
\end{equation}
\end{prop}

\begin{proof}
Let $\bd: w \edge{\xi_1,\dots,\xi_a,\zeta_1,\dots,\zeta_b} u$ 
be the (unique) label-increasing directed path 
from $w$ to $u$ (see Theorem~\ref{thm:inc}), 
where $\xi_{1},\dots,\xi_{a} \in \DJs$ and $\zeta_1,\dots,\zeta_b \in \DJp$. 
Suppose, for a contradiction, that $a \ge 1$. 
Note that $\bd$ is a shortest directed path from $w$ to $u$. 
Since $u \tb{w} v$, it follows that if we define $\bd'$ to be
the concatenation of $\bd$ and a shortest directed path from $u$ to $v$, 
then $\bd'$ is a shortest directed path from $w$ to $v$; 
note that the initial label $\IL{\bd'}$ of $\bd'$ is equal to $\xi_{1} \in \DJs$. 
Here we recall that $\bp_{2}$ is the (unique) label-increasing directed path from $w$ to $v$, 
and $\IL{\bp_{2}} \in \DJp$. By \eqref{eq:ro}, we get $\IL{\bd'} \lhd \IL{\bp_{2}}$, 
which contradicts the fact that $\bp_{2}$ is lexicographically-minimal among the 
shortest directed paths from $w$ to $v$. Thus we get $a = 0$.
Now, because $\bp_{2}$ lies in $\QBG_{(1/2)\vpi_{i}}(W)$, 
it follows from \cite[Lemma 6.7]{LNSSS2} that the concatenation $\bd'$ above 
also lies in $\QBG_{(1/2)\vpi_{i}}(W)$, and hence so does $\bd$. 
Also, we have
\begin{align*}
& \qwt(\bd) + \qwt(\ed(\bd) \Rightarrow x) = 
\qwt(\bp_{2}) - \qwt(u \Rightarrow v) + \qwt(u \Rightarrow x) \\
& \le \qwt(\bp_{2}) - \qwt(u \Rightarrow v) + 
\qwt(u \Rightarrow v) + \qwt(v \Rightarrow x) \quad \text{by Proposition~\ref{prop:qwt}} \\
& = \qwt(\bp_{2}) + \qwt(\ed(\bp_{2}) \Rightarrow x) \le d. 
\end{align*}
Moreover, since both $u=\ed(\bd)$ and $v$ are contained in $x\WJs$, 
it follows from Lemma~\ref{lem:WJs} that $\qwt(u \Rightarrow v) \in Q^{\vee,+}_{\J}$. 
Hence, 
\begin{equation*}
\pair{\vpi_{i}}{\qwt(\bd)} = 
\pair{\vpi_{i}}{\qwt(\bp_{2}) - \qwt(u \Rightarrow v)} = 
\pair{\vpi_{i}}{\qwt(\bp_{2})} = d_{i}. 
\end{equation*}
Therefore we conclude that $\bq:=(\bd,\bt_{u}) \in \bR{w,x,d}$. 
This proves the proposition. 
\end{proof}
%
%
\begin{cor} \label{cor:twp}
Keep the notation and setting above. 
If $\bR{w,x,d} \ne \emptyset$, then $\qwt(w \Rightarrow x) \le d$. 
Therefore, $\langle \CO^{w},\CO_{x} \rangle_{d} = 1$; see Proposition~\ref{prop:twp}. 
\end{cor}

\begin{proof}
Let $\bp_{\min} \in \bR{w,x,d}$ be as in Proposition~\ref{prop:Rtbmin}. 
Since $\ed(\bp_{\min}) = \tbmin{x}{\J}{w}$, and $x \in x\WJs$, 
it follows that $\ed(\bp_{\min}) \tb{w} x$, 
which implies that the concatenation of a shortest directed path from 
$w$ to $\ed(\bp_{\min})$ and a shortest directed path from 
$\ed(\bp_{\min})$ to $x$ is a shortest directed path from $w$ to $x$.
Therefore, 
\begin{equation*}
\qwt(w \Rightarrow x) = 
\underbrace{ \qwt( w \Rightarrow \ed(\bp_{\min}) ) }_{ = \qwt (\bp_{\min}) }  + 
\qwt( \ed(\bp_{\min}) \Rightarrow x ) \le d; 
\end{equation*}
the inequality above follows from the fact that $\bp_{\min} \in \bR{w,x,d}$. 
Thus we have proved the corollary. 
\end{proof}
Combining Proposition~\ref{prop:Rtbmin} and Corollary~\ref{cor:twp}, 
we obtain the following theorem.
%
%
\begin{thm} \label{thm:main1}
Assume that $\vpi_{i}$ is quasi-minuscule. 
Let $w,x \in W$, and $d = \sum_{j \in I} d_{j}\alpha_{j}^{\vee} \in Q^{\vee,+}$ 
(possibly, $d_{i}=0$). If $\# \bR{w,x,d} = 1$, then $\bR{w,x,d} = \bigl\{ \bp_{\min} \bigr\}$, 
where $\bp_{\min}$ is as in Proposition~\ref{prop:Rtbmin}. 
In this case, setting $x_{\min}:=\tbmin{x}{\J}{w}$, we have
\begin{equation} \label{eq:main1a}
\langle \CO^{s_i}, \CO^{w}, \CO_{x} \rangle_{d} =  
\begin{cases}
1 - (-1)^{\ell(w) - \ell(x_{\min})} \be^{-\vpi_{i}+\wt(\eta)}
  & \text{\rm if $\# \bR{w,x,d} = 1$}, \\[1mm]
1 & \text{\rm if $\bR{w,x,d} = \emptyset$ and $\qwt(w \Rightarrow x) \le d$}, \\[1mm]
0 & \text{\rm if $\bR{w,x,d} = \emptyset$ and $\qwt(w \Rightarrow x) \not\le d$}; 
\end{cases}
\end{equation}
recall that $\wt(\eta) = \frac{1}{2}w\vpi_{i}+\frac{1}{2}x\vpi_{i}$. 
\end{thm}

Let us consider the case where $\# \bR{w,x,d} \ge 2$. 
Here we make the following conjecture. 
%
%
\begin{conj} \label{conj:R1even}
Assume that $\vpi_{i}$ is quasi-minuscule. 
We have $\# \bR{w,x,d} \in \{ 1 \} \sqcup 2\BZ_{\ge 0}$ 
for all $w,x \in W$ and $d \in Q^{\vee,+}$. 
\end{conj}

\begin{rem} \label{rem:conj}
As seen in Section~\ref{subsec:mc}, 
if $\vpi_{i}$ is minuscule, then $\# \bR{w,x,d} \in \{0,1\}$ 
for all $w,x \in W$ and $d \in Q^{\vee,+}$. 
Also, by computer, we have checked the following: 
\begin{enu}
\item If $\Fg$ is of type $B_{n}$ with $2 \le n \le 5$, then 
Conjecture~\ref{conj:R1even} is true for all $i \in I$ 
(recall that $\vpi_{n}$ is minuscule); 

\item If $\Fg$ is of type $C_{n}$ with $2 \le n \le 5$, then 
Conjecture~\ref{conj:R1even} is true for all $i \in I$ 
(recall that $\vpi_{1}$ is minuscule); 

\item If $\Fg$ is of type $D_{n}$ with $n = 4,5$, then 
Conjecture~\ref{conj:R1even} is true for all $i \in I$ 
(recall that $\vpi_{1}, \vpi_{n-1}, \vpi_{n}$ are minuscule); 

\item If $\Fg$ is of type $E_{6}$, then 
Conjecture~\ref{conj:R1even} is true for all $i = 1,2,3,5,6$ 
(recall that $\vpi_{1},\vpi_{6}$ are minuscule); 

\item If $\Fg$ is of type $F_{4}$, then 
Conjecture~\ref{conj:R1even} is true for $i=1,4$;

\item If $\Fg$ is of type $G_{2}$, then
Conjecture~\ref{conj:R1even} is true for $i=1$.
\end{enu}
\end{rem}
%
%
\begin{thm} \label{thm:main3}
Assume that $\vpi_{i}$ is quasi-minuscule, and that 
Conjecture~\ref{conj:R1even} is true for $\vpi_{i}$. 
Let $w,x \in W$, and 
$d = \sum_{j \in I} d_{j}\alpha_{j}^{\vee} \in Q^{\vee,+}$ 
(possibly, $d_{i}=0$). If $\# \bR{w,x,d} \ge 2$, then 
%
%
\begin{equation} \label{eq:main3a}
\sum_{ \bp \in \bR{w,x,d} }
 (-1)^{\ell(\bp)} = 0, 
\end{equation}
and hence by Corollary~\ref{cor:twp}, 
\begin{equation}
\langle \CO^{s_i}, \CO^{w}, \CO_{x} \rangle_{d} =  1. 
\end{equation}
\end{thm}

\begin{proof}
By Lemma~\ref{lem:qL} 
(applied to the case where $\J=I \setminus \{i\}$), there exists a sequence 
$j_{1},j_{2},\dots,j_{m}$ of elements in $I_{\af}=I \sqcup \{0\}$ such that 
\begin{equation} \label{eq:qL2}
(s_{j_{k-1}} \cdots s_{j_{1}}w)^{-1}\alpha_{j_{k}} \in \DJp
\qquad \text{\rm and} \qquad
s_{j_{m}} \cdots s_{j_{1}}w \in \WJs. 
\end{equation}
We show the claim by induction on $m$. 

Assume that $m = 0$; in this case, $w \in \WJs$. 
If $\bp=(\bp_{2},\bp_{1}) \in \bR{w,x,d}$, then 
$\eta = \eta_{\bp} = (\mcr{x},\mcr{w}) = (\mcr{x},e) \in \QLS(\vpi_{i})$. 
It follows from \cite[Lemma 3.3]{MNS1} that $\mcr{x}=e$. 
Write $\bp_{2}$ as $\WJs \ni w \edge{\beta_1,\dots,\beta_r} v: = \ed(\bp_2) \in x\WJs = \WJs$ 
with $\beta_1,\dots,\beta_{r} \in \DJp$ such that $\beta_1 \lhd \cdots \lhd \beta_{r}$. 
Since $w,v \in \WJs$, it follows from Lemma~\ref{lem:WJs} that $r = 0$, 
which implies that $w=v$. We conclude that 
\begin{equation}
\bR{w,x,d} = \begin{cases}
 \bigl\{ (\bt_{w},\bt_{w}) \bigr\} 
 & \text{if $x \in \WJs$, $d_{i} = 0$, $\qwt(w \Rightarrow x) \le d$}, \\
 \emptyset & \text{otherwise}, 
 \end{cases}
\end{equation}
which proves the claim in the case where $m=0$. 

Assume that $m > 0$, and $\# \bR{w,x,d} \ge 2$. For simplicity of notation, 
we set $j_{1}:=j$; recall that $w^{-1}\alpha_{j} \in \DJp$. 

\medskip

\paragraph{\bf Case 1.}
%
Assume that $x^{-1}\alpha_{j} \in \DJm$; 
in this case, we see that $y^{-1}\alpha_{j} \in \DJm$ for all $y \in x\WJs$. 
Let $\bp=(\bp_{2},\bp_{1}) \in \bR{w,x,d}$
with $v:=\ed(\bp_{2})=\ed(\bp) \in x\WJs$; recall that $\bp_{1} = \bt_{v}$. 
Write $\bp_{2}$ as 
%
%
\begin{equation} \label{eq:bp2}
\bp_{2} : w \edge{\beta_1,\dots,\beta_r} v
\end{equation}
with $\beta_1,\dots,\beta_{r} \in \DJp$ such that $\beta_1 \lhd \cdots \lhd \beta_{r}$ 
and $\pair{\vpi_{i}}{\beta_{t}^{\vee}}=2$ for all $1 \le t \le r$.
Because $w^{-1}\alpha_{j} \in \DJp$ and $v^{-1}\alpha_{j} \in \DJm$, 
it follows from Lemma~\ref{lem:DL}\,(1) that 
there exists a shortest directed path $\bq_{2}$ 
from $s_{j}w$ to $v$ of the form 
$\bq_{2} : s_{j}w \edge{\beta_1,\dots,\beta_{t-1},\beta_{t+1},\dots,\beta_r} v$ 
for some $1 \le t \le r$, where $\ell(\bq_{2})=\ell(\bp_{2})-1$ and 
$\qwt(\bq_{2}) = \qwt(\bp_{2}) - \delta_{j0} w^{-1}\alpha_{j}^{\vee}$.
Then we deduce that $\Psi(\bp):=(\bq_{2},\bt_{v}) \in \bR{s_{j}w,x,d'}$, where 
%
%
\begin{equation} \label{eq:d'}
d':=d-\delta_{j0} w^{-1}\alpha_{j}^{\vee};
\end{equation} 
note that $\ell(\Psi(\bp)) = \ell(\bp)-1$. We claim that 
\begin{equation} \label{eq:bija}
\text{ the map $\Psi : \bR{w,x,d} \rightarrow 
\bR{s_{j}w,x,d'}$, $\bp \mapsto \Psi(\bp)$, is bijective. }
\end{equation}
Let $\bq':=(\bq_{2}',\bq_{1}') \in \bR{s_{j}w,x,d'}$, 
with $u := \ed(\bq') = \ed(\bq_{2}') \in x\WJs$. 
Recall that $\bq_{1}' = \bt_{u}$. 
Write $\bq_{2}'$ as 
\begin{equation} \label{eq:bq2'}
\bq_{2}' : s_{j} w \edge{\gamma_1,\dots,\gamma_p} u
\end{equation}
with $\gamma_1,\dots,\gamma_{p} \in \DJp$ such that $\gamma_1 \lhd \cdots \lhd \gamma_{p}$ 
and $\pair{\vpi_{i}}{\gamma_{t}^{\vee}}=2$ for all $1 \le t \le p$. 
We have an edge $w \edge{w^{-1}\alpha_{j}} s_{j}w$ by Lemma~\ref{lem:edge}. 
By Lemma~\ref{lem:DL}\,(1), we deduce that the concatenation 
\begin{equation*}
w \edge{w^{-1}\alpha_{j}} \underbrace{ s_{j} w \edge{\gamma_1,\dots,\gamma_p} u }_{=\bq_{2}'}
\end{equation*}
is a shortest directed path from $w$ to $u$; 
in particular, $\ell(w \Rightarrow u) = \ell(\bq_{2}') + 1 = p + 1$ and 
$\qwt(w \Rightarrow u) = \qwt(\bq_{2}')+\delta_{j0}w^{-1}\alpha_{j}^{\vee}$. 
Let $\bp_{2}' : w \edge{\xi_1,\dots,\xi_a,\zeta_1,\dots,\zeta_b} u$ be the (unique)
label-increasing directed path from $w$ to $u$ with respect to the reflection order $\lhd$, 
where $\xi_{1},\dots,\xi_{a} \in \DJs$ and $\zeta_1,\dots,\zeta_b \in \DJp$; 
remark that $a+b = \ell(w \Rightarrow u) = p+1$ and 
$\qwt(\bp_{2}') = \qwt(w \Rightarrow u) = 
 \qwt(\bq_{2}')+\delta_{j0}w^{-1}\alpha_{j}^{\vee}$. We claim that 
\begin{equation} \label{eq:a=0}
a = 0, \qquad \text{and hence} \qquad 
\bp_{2}' : w \edge{\zeta_1,\dots,\zeta_b} u \quad \text{with $b=p+1$}. 
\end{equation}
Indeed, since $w^{-1}\alpha_{j} \in \DJp$ and $u^{-1}\alpha_{j} \in \DJm$, 
it follows from Lemma~\ref{lem:DL}\,(1) that 
there exists a directed path from $s_{j}w$ to $u$ of the form either 
\begin{equation*}
\bd_{2} : s_{j}w \edge{\xi_1,\dots,\xi_{c-1},\xi_{c+1},\dots,\xi_a,\zeta_1,\dots,\zeta_b} u
\quad \text{or} \quad 
\bd_{2}' : s_{j}w \edge{\xi_1,\dots,\xi_a,\zeta_1,\dots,\zeta_{c-1},\zeta_{c+1},\zeta_b} u.
\end{equation*}
By the uniqueness of label-increasing directed paths, 
this directed path is identical to $\bq_{2}' : s_{j} w \edge{\gamma_1,\dots,\gamma_p} u$. 
If $\bq_{2}' = \bd_{2}'$, then we get $a = 0$, as desired. 
If $\bq_{2}' = \bd_{2}$, then we get $a \le 1$. Suppose, for a contradiction, that $a = 1$. 
Then we have $\bp_{2}': w \edge{\xi_1,\zeta_1,\dots,\zeta_b} u$, with $w^{-1}\alpha_{j} = \xi_{1}$. 
However, since $w^{-1}\alpha_{j} \in \DJp$ and $\xi_{1} \in \DJs$, this is a contradiction. 
Thus we get $a = 0$, as desired. We claim that 
\begin{equation} \label{eq:pair=2}
\pair{\vpi_{i}}{\zeta_{c}^{\vee}} = 2 \quad 
 \text{for all $1 \le c \le b$}. 
\end{equation}
Recall from \eqref{eq:bp2} the directed path 
$\bp_{2} : w \edge{\beta_1,\dots,\beta_r} v$. 
Since both $v$ and $u$ are contained in $x \WJs$, 
it follows from Lemma~\ref{lem:WJs} that there exists a directed path 
$v \edge{\phi_1,\dots,\phi_h} u$ from $v$ to $u$ with $\phi_1,\dots,\phi_h \in \DJs$. 
Then the concatenation 
\begin{equation*}
w \edge{\beta_1,\dots,\beta_r} v \edge{\phi_1,\dots,\phi_h} u
\end{equation*}
is a directed path from $w$ to $u$ lying in $\QBG_{(1/2)\vpi_{i}}(W)$. 
Hence it follows from \cite[Lemma 6.7]{LNSSS2} that 
$\bp_{2}': w \edge{\zeta_1,\dots,\zeta_b} u$ also lies in $\QBG_{(1/2)\vpi_{i}}(W)$. 
Since $\pair{\vpi_{i}}{\zeta_{c}^{\vee}} > 0$ and 
$\pair{\vpi_{i}}{\beta^{\vee}} \in \{0,1,2\}$ for all $\beta \in \Delta^{+}$, 
we obtain $\pair{\vpi_{i}}{\zeta_{c}^{\vee}} = 2$ for all $1 \le c \le b$, as desired. 
Therefore, $\Psi'(\bq'):=(\bp_{2}',\bt_{u}) \in \bR{w,x,d'+\delta_{j0}w^{-1}\alpha_{j}^{\vee}} = 
\bR{w,x,d}$; note that $\ell(\Psi'(\bq'))=\ell(\bq')+1$. 
By the uniqueness of label-increasing directed paths, 
we deduce that the map 
$\Psi':\bR{s_{j}w,x,d'} \rightarrow \bR{w,x,d}$, $\bq' \mapsto \Psi'(\bq')$, 
is the inverse map of $\Psi$. Thus we have shown \eqref{eq:bija}. 
By \eqref{eq:bija}, 
$\# \bR{w,x,d} \ge 2$ if and only if $\# \bR{s_{j}w,x,d'} \ge 2$; 
in this case, we compute
\begin{equation*}
\sum_{ \bp \in \bR{w,x,d} }
 (-1)^{\ell(\bp)} =
\sum_{ \bq \in \bR{s_{j}w,x,d'} }
(-1)^{\ell(\Psi^{-1}(\bq))} = 
- \sum_{ \bq \in \bR{s_{j}w,x,d'} }
 (-1)^{\ell(\bq)} \stackrel{ \text{(IH)} }{=} 0.
\end{equation*}

\medskip

\paragraph{\bf Case 2.}
%
Assume that $x^{-1}\alpha_{j} \in \DJp$; 
in this case, we see that $y^{-1}\alpha_{j} \in \DJp$ for all $y \in x\WJs$. 
Let $\bp=(\bp_{2},\bp_{1}) \in \bR{w,x,d}$
with $v:=\ed(\bp_{2})=\ed(\bp) \in x\WJs$; recall that $\bp_{1} = \bt_{v}$. 
Write $\bp_{2}$ as
\begin{equation} \label{eq:bp2a}
\bp_{2} : w = u_{0} \edge{\beta_1} u_1 \edge{\beta_2} \cdots \edge{\beta_r} u_r = v, 
\end{equation}
with $\beta_1,\dots,\beta_{r} \in \DJp$ such that $\beta_1 \lhd \cdots \lhd \beta_{r}$ 
and $\pair{\vpi_{i}}{\beta_{t}^{\vee}}=2$ for all $1 \le t \le r$.
Because $w^{-1}\alpha_{j} \in \DJp$ and $v^{-1}\alpha_{j} \in \DJp$, 
it follows from Lemma~\ref{lem:DL}\,(3) that 
there exists a (shortest) directed path $\bd_{2}$ 
from $s_{j}w$ to $s_{j}v$ such that $\ell(\bd_{2})=\ell(\bp_{2})$ and 
$\qwt(\bd_{2}) = \qwt(\bp_{2}) - \delta_{j0} w^{-1}\alpha_{j}^{\vee} + \delta_{j0} v^{-1}\alpha_{j}^{\vee}$.

Now, let 
%
%
\begin{equation} \label{eq:bq2b}
\bq_{2} : s_{j}w = v_{0} \edge{\xi_1} \cdots \edge{\xi_a} v_{a} 
\edge{\zeta_1} \cdots \edge{\zeta_b} v_{a+b} = s_{j}v.
\end{equation}
be the (unique) label-increasing directed path from $s_{j}w$ to $s_{j}v$ 
with respect to the reflection order $\lhd$, 
where $\xi_{1},\dots,\xi_{a} \in \DJs$ and $\zeta_1,\dots,\zeta_b \in \DJp$; 
note that $a+b=\ell(s_{j}w \Rightarrow s_{j}v) = \ell(w \Rightarrow v) = r$. 
We claim that 
\begin{equation} \label{eq:a0-2}
a = 0 \quad \text{and} \quad 
\pair{\vpi_{i}}{\zeta_{c}^{\vee}}=2 
\text{ for all $1 \le c \le b$}. 
\end{equation}
Assume first that $v_{t}^{-1}\alpha_{j} \in \Delta^{-}$ for all $0 \le t \le a+b = r$. 
By Lemma~\ref{lem:DL}\,(3), 
we have a shortest directed path from $w$ to $v$ of the form
\begin{equation} \label{eq:bp2b}
\bp_{2}' : w = s_{j}v_{0} \edge{\xi_1} \cdots \edge{\xi_a} s_{j}v_{a} 
\edge{\zeta_1} \cdots \edge{\zeta_b} s_{j}v_{a+b} = v. 
\end{equation}
Recall from \eqref{eq:bp2a} the label-increasing directed path $\bp_{2}$. 
By the uniqueness of label-increasing directed paths, 
it follows that $\bp_{2} = \bp_{2}'$, which implies \eqref{eq:a0-2}; 
in particular, $u_{t}^{-1}\alpha_{j} \in \Delta^{+}$ for all $0 \le t \le r = a+b$. 
Similarly, we can deduce that if 
$u_{t}^{-1}\alpha_{j} \in \Delta^{+}$ for all $0 \le t \le r = a+b$, then 
$v_{t}^{-1}\alpha_{j} \in \Delta^{-}$ for all $0 \le t \le a+b = r$. 
Assume next that $v_{t}^{-1}\alpha_{j} \in \Delta^{+}$ for some $0 \le t \le a+b = r$, 
or equivalently, $u_{t}^{-1}\alpha_{j} \in \Delta^{-}$ for some $0 \le t \le r = a+b$; 
notice that $0 < t < a+b = r$. 
Suppose, for a contradiction, that $a \ge 1$. 
By Lemma~\ref{lem:DL}\,(2) (applied to the directed subpath 
$v_{t} \edge{\bullet} \cdots \edge{\bullet} v_{a+b}=s_{j}v$ in $\bq_{2}$), 
there exists a directed path from $s_{j}w$ to $v$ 
whose initial edge is equal to $\xi_{1} \in \DJs$. 
Also, by Lemma~\ref{lem:DL}\,(1) 
(applied to the directed subpath 
$w = u_{0} \edge{\bullet} \cdots \edge{\bullet} u_{t}$ in 
$\bp_{2}$), there exists a directed path 
from $s_{j}w$ to $v$ of the form 
$s_{j}w \edge{\beta_1,\dots,\beta_{t'-1},\beta_{t'+1},\dots,\beta_r} v$ 
for some $1 \le t' \le r-1$. 
By the uniqueness of label-increasing directed paths, 
we get $\xi_{1} \in \DJp$, which is a contradiction. 
Thus we get $a=0$. The concatenation 
\begin{equation*}
w \edge{w^{-1}\alpha_{j}} s_{j}w \edge{\beta_1,\dots,\beta_{t-1},\beta_{t+1},\dots,\beta_r} v
\end{equation*}
is a shortest directed path from $w$ to $v$. 
Recall that $w^{-1}\alpha_{j} \in \DJp$, and hence $\pair{\vpi_{i}}{w^{-1}\alpha_{j}^{\vee}} > 0$. 
Since $\bp_{2}$ lies in $\QBG_{(1/2)\vpi_{i}}(W)$, 
it follows from \cite[Lemma 6.7]{LNSSS2} that the directed path above also lies in 
$\QBG_{(1/2)\vpi_{i}}(W)$; in particular, $\pair{\vpi_{i}}{w^{-1}\alpha_{j}^{\vee}} = 2$. 
Since $\wt(\eta) = \frac{1}{2}x\vpi_{i} + \frac{1}{2}w\vpi_{i} \in \Lambda$, 
we see that
\begin{equation*}
\BZ \ni \pair{\wt(\eta)}{\alpha_{j}^{\vee}} =
\frac{1}{2}\pair{x\vpi_{i}}{\alpha_{j}^{\vee}} + 
\frac{1}{2}\pair{w\vpi_{i}}{\alpha_{j}^{\vee}} = 
\frac{1}{2}\pair{v\vpi_{i}}{\alpha_{j}^{\vee}} + 1, 
\end{equation*}
and hence $\pair{\vpi_{i}}{v^{-1}\alpha_{j}^{\vee}} = 2$; 
recall that $v^{-1}\alpha_{j} \in \DJp$. 
We see that the concatenation 
\begin{equation*}
s_{j}w \edge{\beta_1,\dots,\beta_{t-1},\beta_{t+1},\dots,\beta_r} v \edge{v^{-1}\alpha_{j}} s_{j}v
\end{equation*}
is a shortest directed path from $s_{j}w$ to $s_{j}v$ 
lying in $\QBG_{(1/2)\vpi_{i}}(W)$. 
Hence it follows from \cite[Lemma 6.7]{LNSSS2} that 
$\bq_{2} : s_{j}w \edge{\zeta_1,\dots,\zeta_b} s_{j}v$ also lies in $\QBG_{(1/2)\vpi_{i}}(W)$. 
In both cases, we set $\Phi(\bp) := (\bq_{2}, \bt_{s_{j}v})$; we have 
\begin{equation*}
\ell(\Phi(\bp)) = \ell(\bp), \qquad 
\qwt(\Phi(\bp)) = \qwt(\bp)  - \delta_{j0} w^{-1}\alpha_{j}^{\vee} + \delta_{j0} v^{-1}\alpha_{j}^{\vee},
\end{equation*}
\begin{equation*}
\qwt(\underbrace{ \ed(\Phi(\bp)) }_{= s_{j}v} \Rightarrow s_{j}x) =
\qwt(\underbrace{ \ed(\bp) }_{=v} \Rightarrow x) 
 - \delta_{j0} v^{-1}\alpha_{j}^{\vee} + \delta_{j0} x^{-1}\alpha_{j}^{\vee} \quad \text{by Lemma~\ref{lem:DL}\,(3)}.
\end{equation*}
Therefore we conclude that $\Phi(\bp) \in \bR{s_{j}w,s_{j}x,d''}$, where 
\begin{equation} \label{eq:d''}
d'':= d - \delta_{j0} w^{-1}\alpha_{j}^{\vee} + \delta_{j0} x^{-1}\alpha_{j}^{\vee}.
\end{equation}
We claim that 
\begin{equation} \label{eq:bijb}
\text{ the map $\Phi : \bR{w,x,d} \rightarrow \bR{s_{j}w,s_{j}x,d''}$, 
$\bp \mapsto \Phi(\bp)$, is bijective. }
\end{equation}
Let $\bq'=(\bq_2',\bq_1') \in \bR{s_{j}w,s_{j}x,d''}$, 
with $u := \ed(\bq_2') = \ed(\bq') \in s_{j}x \WJs$; recall that $\bq_1' = \bt_{u}$. 
By the same argument, we can show that 
the label-increasing directed path $\bp_{2}'$ from $w$ to $s_{j}u$ is of the form
\begin{equation*}
\bp_{2}' : w \edge{\phi_1,\dots,\phi_r} s_{j}u \in x\WJs, 
\end{equation*}
where $\phi_1,\dots,\phi_{r} \in \DJp$ such that $\phi_1 \lhd \cdots \lhd \phi_{r}$ 
and $\pair{\vpi_{i}}{\phi_{t}^{\vee}}=2$ for all $1 \le t \le r$, 
satisfying the conditions that $\ell(\bp_{2}') = \ell(\bq_{2}')$ and 
$\qwt(\bp_{2}') = \qwt(\bq_{2}') - \delta_{j0} (s_{j}w)^{-1}\alpha_{j}^{\vee} + 
\delta_{j0}u^{-1}\alpha_{j}^{\vee}$. 
Hence we deduce in exactly the same way as above 
that $\Phi'(\bq'):=(\bp_{2}',\bt_{s_{j}u}) \in \bR{w,x,d}$. 
By the uniqueness of label-increasing directed paths, the map 
$\Phi':\bR{s_{j}w,s_{j}x,d''} \rightarrow \bR{w,x,d}$, $\bq' \mapsto \Phi'(\bq')$, is 
the inverse map of $\Phi$ above. Thus we have shown \eqref{eq:bijb}. 
By \eqref{eq:bijb}, 
$\# \bR{w,x,d} \ge 2$ if and only if $\# \bR{s_{j}w,s_{j}x,d''} \ge 2$; 
in this case, we compute
\begin{equation*}
\sum_{ \bp \in \bR{w,x,d} }
 (-1)^{\ell(\bp)} =
\sum_{ \bq \in \bR{s_{j}w,s_{j}x,d''} }
(-1)^{\ell(\Phi^{-1}(\bq))} = 
\sum_{ \bq \in \bR{s_{j}w,s_{j}x,d''} }
 (-1)^{\ell(\bq)} \stackrel{ \text{(IH)} }{=} 0.
\end{equation*}

\medskip

\paragraph{\bf Case 3.}
%
Assume that $x^{-1}\alpha_{j} \in \Delta_{\J}$; 
in this case, $x\WJs = s_{j}x\WJs$, and 
$y^{-1}\alpha_{j} \in \Delta_{\J}$ for all $y \in x \WJs$. 
For $w',x' \in W$ and $\delta \in Q^{\vee,+}$, we set 
\begin{equation*}
(\bR{w',x',\delta})^{\pm} := \bigl\{ \bp \in \bR{w',x',\delta} \mid 
(\ed(\bp))^{-1}\alpha_{j} \in \Delta^{\pm} \bigr\}. 
\end{equation*}

\medskip

\paragraph{\bf Subcase 3.1.}
%
Assume that $x^{-1}\alpha_{j} \in \DJsm$. 
First, let $\bp = ( \bp_{2},\bp_{1}) \in (\bR{w,x,d})^{+}$, 
and set $v:=\ed(\bp) = \ed(\bp_{2}) \in x \WJs$; 
recall that $v^{-1}\alpha_{j} \in \DJs$. 
By the same argument as in Case 2, we deduce that 
the label-increasing directed path $\bq_{2}$ 
from $s_{j}w$ to $s_{j}v$ is of the form 
$\bq_{2} : s_{j}w \edge{\zeta_{1},\dots,\zeta_{r}} s_{j}v$,
where $\zeta_1,\dots,\zeta_{r} \in \DJp$ 
such that $\zeta_1 \lhd \cdots \lhd \zeta_{r}$ 
and $\pair{\vpi_{i}}{\zeta_{t}^{\vee}}=2$ for all $1 \le t \le r$; 
note that $\ell(\bq_{2})= \ell(\bp_{2})$ and 
$\qwt(\bq_{2}) = \qwt( s_{j}w \Rightarrow s_{j}v) = 
\qwt(w \Rightarrow v) - \delta_{j0}w^{-1}\alpha_{j}^{\vee}+\delta_{j0}v^{-1}\alpha_{j}^{\vee}$.
Also, we see by Lemma~\ref{lem:DL}\,(1) that 
$\qwt(s_{j}v \Rightarrow x) = \qwt(v \Rightarrow x) - \delta_{j0} v^{-1}\alpha_{j}^{\vee}$. 
Hence, $\Phi(\bp):=(\bq_{2},\bt_{s_{j}v}) \in (\bR{s_{j}w,x,d'})^{-}$, where 
$d'=d - \delta_{j0}w^{-1}\alpha_{j}^{\vee}$; remark that $\ell(\Phi(\bp)) = \ell(\bp)$. 
As in Case 2, we can show that the map 
$\Phi:(\bR{w,x,d})^{+} \rightarrow (\bR{s_{j}w,x,d'})^{-}$, $\bp \mapsto \Phi(\bp)$,  
is bijective. 

Next, let $\bp = ( \bp_{2},\bp_{1}) \in (\bR{w,x,d})^{-}$, 
and set $u := \ed(\bp) = \ed(\bp_{2}) \in x \WJs$; 
note that $u^{-1}\alpha_{j} \in \DJsm$. 
Write $\bp_{2}$ as $\bp_{2}:w \edge{\beta_{1},\dots,\beta_{r}} u$. 
By Lemma~\ref{lem:DL}\,(1), there exists a directed path from $s_{j}w$ to $u$ 
of the form $\bq_{2}:s_{j}w \edge{\beta_{1},\dots,\beta_{t-1},\beta_{t+1},\dots,\beta_{r}} u$. 
Then we see that $\Psi(\bp):=(\bq_{2},\bt_{u}) \in (\bR{s_{j}w,x,d'})^{-}$; 
note that $\ell(\Psi(\bp)) = \ell(\bp) -1$. 
We can show that the map $\Psi:(\bR{w,x,d})^{-} \rightarrow (\bR{s_{j}w,x,d'})^{-}$, 
$\bp \mapsto \Psi(\bp)$, is bijective. 

Using the bijections $\Phi$ and $\Psi$ above, we compute
\begin{align*}
& \sum_{ \bp \in \bR{w,x,d} }
  (-1)^{\ell(\bp)} =
\sum_{ \bp \in (\bR{w,x,d})^{+} }
  (-1)^{\ell(\bp)} + 
\sum_{ \bp \in (\bR{w,x,d})^{-} }
  (-1)^{\ell(\bp)} \\[3mm]
& = \sum_{ \bq \in (\bR{s_{j}w,x,d'})^{-} }
  (-1)^{\ell(\Phi^{-1}(\bq))} + 
\sum_{ \bq \in (\bR{s_{j}w,x,d'})^{-} }
  (-1)^{\ell(\Psi^{-1}(\bq))} \\[3mm]
& = \sum_{ \bq \in (\bR{s_{j}w,x,d'})^{-} }
  (-1)^{\ell(\bq)} -
\sum_{ \bq \in (\bR{s_{j}w,x,d'})^{-} }
  (-1)^{\ell(\bq)} = 0. 
\end{align*}

\medskip

\paragraph{\bf Subcase 3.2.}
%
Assume that $x^{-1}\alpha_{j} \in \DJs$. 
Define $d'$ and $d''$ as in \eqref{eq:d'} and \eqref{eq:d''}, respectively. 
By the same argument as above, we can show that
\begin{itemize}
\item there exists a bijection 
$\Phi:(\bR{w,x,d})^{+} \rightarrow (\bR{s_{j}w,s_{j}x,d''})^{-}$ 
satisfying the condition that $\ell(\Phi(\bp)) = \ell(\bp)$ for 
$\bp \in (\bR{w,x,d})^{+}$ (see also Case 2); 

\item there exists a bijection 
$\Psi:(\bR{w,x,d})^{-} \rightarrow (\bR{s_{j}w,x,d'})^{-}$ 
satisfying the condition that $\ell(\Psi(\bp)) = \ell(\bp)-1$ 
for $\bp \in (\bR{w,x,d})^{-}$ (see also Case 1). 
\end{itemize}
In addition, we have 
\begin{equation} \label{eq:claim32a}
(\bR{s_{j}w,s_{j}x,d''})^{-} \supset (\bR{s_{j}w,x,d'})^{-}, 
\end{equation}
\begin{equation} \label{eq:claim32b}
(\bR{s_{j}w,s_{j}x,d''})^{+} = (\bR{s_{j}w,x,d'})^{+}. 
\end{equation}
Indeed, let us show that 
$(\bR{s_{j}w,s_{j}x,d''})^{\pm} \supset (\bR{s_{j}w,x,d'})^{\pm}$. 
Let $\bq = (\bq_{2},\bq_{1}) \in (\bR{s_{j}w,x,d'})^{\pm}$, 
and set $z := \ed(\bq) \in x\WJs$. 
Since $x^{-1}\alpha_{j} \in \DJs$ by the assumption of Subcase 3.2, 
there exists an edge $x \edge{x^{-1}\alpha_{j}} s_{j}x$; 
note that $\qwt (x \edge{x^{-1}\alpha_{j}} s_{j}x) = \delta_{j0}x^{-1}\alpha_{j}^{\vee}$. 
Because the concatenation of a shortest directed path from $z = \ed(\bq)$ to $x$ 
and the edge $x \edge{x^{-1}\alpha_{j}} s_{j}x$ is a directed path from $z$ to $s_{j}x$, 
it follows from Proposition~\ref{prop:qwt} that
\begin{equation} \label{eq:sc32b}
\qwt ( \ed(\bq) \Rightarrow s_{j}x ) \le 
\qwt ( \ed(\bq) \Rightarrow x ) + \delta_{j0}x^{-1}\alpha_{j}^{\vee}. 
\end{equation}
Since $\bq \in \bR{s_{j}w,x,d'}$, we have 
$\qwt ( \ed(\bq) \Rightarrow x ) \le d' - \qwt(\bq)$. Combining these inequalities, 
we obtain $\qwt ( \ed(\bq) \Rightarrow s_{j}x ) \le d'' - \qwt(\bq)$, 
which implies that $\bq \in \bQLS{s_{j}w,s_{j}x,d''}$. 
Since $x^{-1}\alpha_{j} \in \Delta_{\J}$, we see that 
$\pair{\vpi_{i}}{d''}=\pair{\vpi_{i}}{d'}$. 
Also, since $\bq \in \bR{s_{j}w,x,d'}$, we have 
$\pair{\vpi_{i}}{d'-\qwt(\bq)} = 0$. Combining these equalities, 
we get $\pair{\vpi_{i}}{d''-\qwt(\bq)} = 0$. 
Because $\bq = (\bq_{2},\bq_{1}) \in (\bR{s_{j}w,x,d'})^{\pm}$, 
it follows that $\ell(\bq_{1}) = 0$, $\ed(\bq) \in x\WJs = s_{j}x\WJs$, 
and $(\ed(\bq))^{-1}\alpha_{j} \in \Delta^{\pm}$. 
Thus we conclude that $\bq \in (\bR{s_{j}w,s_{j}x,d''})^{\pm}$. 
For the inclusion $(\bR{s_{j}w,s_{j}x,d''})^{+} \subset (\bR{s_{j}w,x,d'})^{+}$, 
let $\bq = (\bq_{2},\bq_{1}) \in (\bR{s_{j}w,s_{j}x,d''})^{+}$. 
Because $(\ed(\bq))^{-1}\alpha_{j} \in \Delta^{+}$ and 
$(s_{j}x)^{-1} \alpha_{j} \in \Delta^{-}$, we deduce by Lemma~\ref{lem:DL}\,(2) that 
$\bq = (\bq_{2},\bq_{1}) \in (\bR{s_{j}w,x,d'})^{+}$. 
Thus we have shown \eqref{eq:claim32a} and \eqref{eq:claim32b}. 

Using the bijections $\Phi$, $\Psi$ and 
\eqref{eq:claim32a}, \eqref{eq:claim32b}, we compute
\begin{align}
& \sum_{ \bp \in \bR{w,x,d} }
  (-1)^{\ell(\bp)} =
\sum_{ \bp \in (\bR{w,x,d})^{+} }
  (-1)^{\ell(\bp)} + 
\sum_{ \bp \in (\bR{w,x,d})^{-} }
  (-1)^{\ell(\bp)} \nonumber \\[3mm]
& = \sum_{ \bq \in (\bR{s_{j}w,s_{j}x,d''})^{-} }
  (-1)^{\ell(\bq)} - 
  \sum_{ \bq \in (\bR{s_{j}w,x,d'})^{-} }
  (-1)^{\ell(\bq)} \nonumber \\[3mm]
& = \sum_{ \bq \in (\bR{s_{j}w,s_{j}x,d''})^{-} }
  (-1)^{\ell(\bq)} - 
  \sum_{ \bq \in (\bR{s_{j}w,x,d'})^{-} }
  (-1)^{\ell(\bq)} \nonumber \\
& \hspace*{40mm}
  + \sum_{ \bq \in (\bR{s_{j}w,x,d'})^{+} }
  (-1)^{\ell(\bq)} 
  - \sum_{ \bq \in (\bR{s_{j}w,x,d'})^{+} }
  (-1)^{\ell(\bq)} \nonumber \\[3mm]
& = \sum_{ \bq \in (\bR{s_{j}w,s_{j}x,d''})^{-} }
  (-1)^{\ell(\bq)} - 
  \sum_{ \bq \in (\bR{s_{j}w,x,d'})^{-} }
  (-1)^{\ell(\bq)} \nonumber \\
& \hspace*{40mm}
  + \sum_{ \bq \in (\bR{s_{j}w,s_{j}x,d''})^{+} }
  (-1)^{\ell(\bq)} 
  - \sum_{ \bq \in (\bR{s_{j}w,x,d'})^{+} }
  (-1)^{\ell(\bq)} \nonumber \\[3mm]
& = 
  \sum_{ \bq \in \bR{s_{j}w,s_{j}x,d''} }
  (-1)^{\ell(\bq)} - 
  \sum_{ \bq \in \bR{s_{j}w,x,d'} }
  (-1)^{\ell(\bq)}. \label{eq:c32a}
\end{align}
Now, assume that $\# \bR{s_{j}w,x,d'} \ge 2$. 
Since $\# \bR{s_{j}w,s_{j}x,d''} \ge \# \bR{s_{j}w,x,d'} \ge 2$ 
by \eqref{eq:claim32a} and \eqref{eq:claim32b}, 
it follows by the induction hypothesis and \eqref{eq:c32a} that 
$\sum_{ \bp \in \bR{w,x,d} } (-1)^{\ell(\bp)} = 0$. 

Assume that $\# \bR{s_{j}w,x,d'} = 0$. 
Using the bijection 
$\Psi:(\bR{w,x,d})^{-} \rightarrow (\bR{s_{j}w,x,d'})^{-}$, 
we obtain $\# (\bR{w,x,d})^{-} = 0$. 
Since $\# \bR{w,x,d} \ge 2$ by assumption, 
we get $\# (\bR{s_{j}w,s_{j}x,d''})^{-} = \# (\bR{w,x,d})^{+} \ge 2$ by 
using the bijection 
$\Phi:(\bR{w,x,d})^{+} \rightarrow (\bR{s_{j}w,s_{j}x,d''})^{-}$. 
Hence, $\#\bR{s_{j}w,s_{j}x,d''} \ge 2$. 
Thus it follows by the induction hypothesis and \eqref{eq:c32a} that 
$\sum_{ \bp \in \bR{w,x,d} } (-1)^{\ell(\bp)} = 0$. 

Assume that $\# \bR{s_{j}w,x,d'} = 1$. 
Let $\bp_{\min} \in \bR{w,x,d} \ne \emptyset$ and 
$\bq_{\min} \in \bR{s_{j}w,x,d'} \subset \bR{s_{j}w,s_{j}x,d''}$ be such that 
\begin{equation*}
\ed(\bp_{\min}) = \tbmin{x}{\J}{w}=:v, \qquad 
\ed(\bq_{\min}) = \tbmin{x}{\J}{s_{j}w} =:u,
\end{equation*}
respectively; see Proposition~\ref{prop:Rtbmin}. 
We claim that 
\begin{equation} \label{eq:c4-c1}
\text{$v^{-1}\alpha_{j} \in \DJs$, and hence $\bp_{\min} \in (\bR{w,x,d})^{+}$}. 
\end{equation}
\begin{equation} \label{eq:c4-c2}
\text{if $u \ne s_{j}v$, then $u^{-1}\alpha_{j} \in \DJs$, 
and hence $\bq_{\min} \in (\bR{s_{j}w,x,d'})^{+} = (\bR{s_{j}w,s_{j}x,d''})^{+}$}. 
\end{equation}
First, let us show \eqref{eq:c4-c1}. Suppose, for a contradiction, that 
$v^{-1}\alpha_{j} \in \DJsm$. By Lemma~\ref{lem:DL}\,(2) and Lemma~\ref{lem:edge}, 
there exists a shortest directed path from $w$ to $v$ passing through $s_{j}v$, 
which implies that $s_{j}v \tb{w} v$. 
Because $s_{j}v \in x\WJs$, this contradicts the minimality 
of $v = \tbmin{x}{\J}{w}$. Thus we have shown \eqref{eq:c4-c1}. 
Next, let us show \eqref{eq:c4-c2}. Assume that 
$u^{-1}\alpha_{j} \in \DJsm$. By the minimality of $u = \tbmin{x}{\J}{s_{j}w}$, there exists a 
shortest directed path from $s_{j}w$ to $s_{j}v \in x\WJs$ passing through $u$. 
Note that all of $(s_{j}w)^{-1}\alpha_{j}$, $(s_{j}v)^{-1}\alpha_{j}$, and $u^{-1}\alpha_{j}$ 
are contained in $\Delta^{-}$. We see from Lemma~\ref{lem:DL}\,(3) that there exists a 
shortest directed path from $w$ to $v$ passing through $s_{j}u \in x\WJs$, 
which implies that $s_{j}u \tb{w} v$. Because $v \tb{w} s_{j}u$ 
by the minimality of $v = \tbmin{x}{\J}{w}$, we get $v = s_{j}u$. 
Thus we have shown \eqref{eq:c4-c2}. 

Suppose, for a contradiction, that $u \ne s_{j}v$; 
recall that $\#\bR{w,x,d} \ge 2$ and $\# \bR{s_{j}w,x,d'} = 1$. 
We have $(\bR{s_{j}w,s_{j}x,d''})^{+} = (\bR{s_{j}w,x,d'})^{+} = \bigl\{ \bq_{\min} \bigr\}$ 
by \eqref{eq:c4-c2}, and 
$\#(\bR{w,x,d})^{-} = \#(\bR{s_{j}w,x,d'})^{-} = 0$ 
by using the bijection $\Psi$. Since $\#\bR{w,x,d} \ge 2$, we get 
$\#(\bR{s_{j}w,s_{j}x,d''})^{-} = \#(\bR{w,x,d})^{+} = \#\bR{w,x,d}-
\#(\bR{w,x,d})^{-} =\#\bR{w,x,d} \ge 2$. 
Since Conjecture~\ref{conj:R1even} is assumed to hold, 
it follows that $\# (\bR{s_{j}w,s_{j}x,d''})^{-} \in 2\BZ_{\ge 1}$. 
Thus we obtain $\# \bR{s_{j}w,s_{j}x,d''} = 
\# (\bR{s_{j}w,s_{j}x,d''})^{+} + \# (\bR{s_{j}w,s_{j}x,d''})^{-} 
\in 1 + 2\BZ_{\ge 1}$, which contradicts the assumption that 
Conjecture~\ref{conj:R1even} is true. Thus we obtain $u = s_{j}v$, 
which implies that $(\bR{s_{j}w,x,d'})^{-} = \bigl\{ \bq_{\min} \bigr\}$ and 
$\# (\bR{s_{j}w,x,d'})^{+} = 0$. We see that 
$\# (\bR{s_{j}w,s_{j}x,d''})^{+} = \# (\bR{s_{j}w,x,d'})^{+} = 0$ and 
$\# (\bR{s_{j}w,s_{j}x,d''})^{-} \ge \# (\bR{s_{j}w,x,d'})^{-} = 1$. 
Suppose, for a contradiction, that 
$\# (\bR{s_{j}w,s_{j}x,d''})^{-} \ge 2$. 
Since $\# \bR{s_{j}w,s_{j}x,d''} = 
\# (\bR{s_{j}w,s_{j}x,d''})^{+} + \# (\bR{s_{j}w,s_{j}x,d''})^{-} = 
\# (\bR{s_{j}w,s_{j}x,d''})^{-}$, we obtain 
$\# (\bR{s_{j}w,s_{j}x,d''})^{-} \in 2\BZ_{\ge 1}$ by the assumption that 
Conjecture~\ref{conj:R1even} is true. 
Hence, 
$\# \bR{w,x,d} = 
\# (\bR{w,x,d})^{+} + \# (\bR{w,x,d})^{-} = 
\# (\bR{s_{j}w,s_{j}x,d''})^{-} + \# (\bR{s_{j}w,x,d'})^{-} \in 2\BZ_{\ge 1} + 1$, 
which contradicts the assumption that 
Conjecture~\ref{conj:R1even} is true. 
Thus we obtain $\# (\bR{s_{j}w,s_{j}x,d''})^{-} = 1$. 
Combining this equality and \eqref{eq:claim32a}, \eqref{eq:claim32b}, 
we conclude that $\bR{s_{j}w,s_{j}x,d''} = \bR{s_{j}w,x,d'}$. 
Therefore, $\sum_{ \bp \in \bR{w,x,d} } (-1)^{\ell(\bp)} = 0$ by \eqref{eq:c32a}. 

This completes the proof of Theorem~\ref{thm:main3}. 
\end{proof}

\subsection{Variation and its application.}
\label{subsec:variation}

Let $i \in I$ be such that $\vpi_{i}$ is quasi-minuscule. 
For $w,x \in W$ and $k \in \BZ_{\ge 0}$, we set
\begin{equation}
\hbR{w,x\WJs,k}:=\left\{ \bp=(\bp_{2},\bp_{1}) \in \bQLS{w} \ \Biggm| \ 
\begin{array}{l}
\ell(\bp_{1})=0, \, \pair{\vpi_{i}}{\qwt(\bp_{2})}=k, \\[2mm]
\ed(\bp) = \ed(\bp_{2}) \in x\WJs
\end{array} \right\};
\end{equation}
we should remark that for $w,x \in W$ and 
$d=\sum_{j \in I} d_{j}\alpha_{j}^{\vee} \in Q^{\vee,+}$, 
\begin{equation} \label{eq:RvshR}
\bR{w,x,d} = \bigl\{ 
\bp \in \hbR{w,x\WJs,d_{i}} \mid 
\qwt(\bp) + \qwt ( \ed(\bp) \Rightarrow x) \le d \bigr\}.
\end{equation}
We can prove the following proposition in exactly the same way as 
Proposition~\ref{prop:Rtbmin}. Here we remark that 
$\pair{\vpi_{i}}{\qwt(u \Rightarrow v)} = 0$, or equivalently, 
$\qwt(u \Rightarrow v) \in \QJvp$ since $u,v \in x\WJs$.
%
%
\begin{prop} \label{prop:Rtmin2}
Keep the notation and setting above. 
Assume that $\hbR{w,x\WJs,k} \ne \emptyset$. 
Let $v \in x\WJs$ be such that $v = \ed(\bp) = \ed(\bp_{2})$ for some (unique) 
$\bp=(\bp_{2},\bp_{1}) \in \hbR{w,x\WJs,k}$, and let $u \in x\WJs$ be such that $u \tb{w} v$ 
with respect to the $w$-tilted Bruhat order $\tb{w}$. 
Then there exists a unique $\bq = (\bq_{2},\bq_{1}) \in \hbR{w,x\WJs,k}$ such that 
$u = \ed(\bq) = \ed(\bq_{2})$. Therefore there exists a unique 
$\bp_{\min} \in \hbR{w,x\WJs,k}$ such that $\ed(\bp_{\min}) = \tbmin{x}{\J}{w}$.
\end{prop}
%
%
\begin{thm}[cf. Conjecture~\ref{conj:R1even}] \label{thm:hR1even}
Let $i \in I$ be such that $\vpi_{i}$ is quasi-minuscule, 
and let $w,x \in W$, $k \in \BZ_{\ge 0}$. 
We have $\# \hbR{w,x\WJs,k} \in \{1\} \sqcup 2\BZ_{\ge 0}$. 
Moreover, if $\# \hbR{w,x\WJs,k} \in 2\BZ_{\ge 0}$, then 
\begin{equation} \label{eq:hR1}
\sum_{\bp \in \hbR{w,x\WJs,k}} (-1)^{\ell(\bp)} = 0.
\end{equation}
\end{thm}

\begin{proof}
By Lemma~\ref{lem:qL} (applied to the case where $\J=I \setminus \{i\}$), 
there exists a sequence $j_{1},j_{2},\dots,j_{m}$ of elements 
in $I_{\af}=I \sqcup \{0\}$ satisfying \eqref{eq:qL2}. 
We show the claim by induction on $m$. 
If $m = 0$, then we can show in exactly the same way as 
in the proof of Theorem~\ref{thm:main3} that 
\begin{equation*}
\hbR{w,x\WJs,k} = \begin{cases}
 \bigl\{ (\bt_{w},\bt_{w}) \bigr\} 
 & \text{if $x \in \WJs$}, \\
 \emptyset & \text{otherwise}, 
 \end{cases}
\end{equation*}
which proves the claims in the case where $m=0$. 

Assume that $m > 0$, and $\# \hbR{w,x\WJs,k} \ge 2$. 
For simplicity of notation, 
we set $j_{1}:=j$; recall that $w^{-1}\alpha_{j} \in \DJp$. 

\medskip

\paragraph{\bf Case 1'.}
%
Assume that $x^{-1}\alpha_{j} \in \DJm$. We can verify, 
as in Case 1 in the proof of Theorem~\ref{thm:main3}, that 
there exists a bijection 
\begin{equation*}
\Psi:\hbR{w,x\WJs,k} \rightarrow \hbR{s_{j}w,x\WJs,k'}
\end{equation*}
such that $\ell(\Psi(\bp)) = \ell(\bp)-1$ for 
$\bp \in \hbR{w,x\WJs,k}$, where 
\begin{equation} \label{eq:m'}
k':=k-\delta_{j0}\pair{\vpi_{i}}{w^{-1}\alpha_{j}^{\vee}}.
\end{equation}
Notice that $\# \hbR{s_{j}w,x\WJs,k'} = \# \hbR{w,x\WJs,k} \ge 2$. 
By the induction hypothesis, $\# \hbR{s_{j}w,x\WJs,k'} \in 2\BZ_{\ge 1}$, 
and hence $\# \hbR{w,x\WJs,k} \in 2\BZ_{\ge 1}$. Also, 
by using this bijection $\Psi$ and the induction hypothesis, 
we compute
\begin{equation*}
\sum_{ \bp \in \hbR{w,x\WJs,k} }
 (-1)^{\ell(\bp)} =
- \sum_{ \bq \in \hbR{s_{j}w,x\WJs,k'} }
 (-1)^{\ell(\bq)} \stackrel{ \text{(IH)} }{=} 0. 
\end{equation*}

\medskip

\paragraph{\bf Case 2'.}
%
Assume that $x^{-1}\alpha_{j} \in \DJp$. We can verify, 
as in Case 2 in the proof of Theorem~\ref{thm:main3}, that 
there exists a bijection 
\begin{equation*}
\Phi:\hbR{w,x\WJs,k} \rightarrow \hbR{s_{j}w,s_{j}x\WJs,k''}
\end{equation*}
such that $\ell(\Phi(\bp)) = \ell(\bp)$ for 
$\bp \in \hbR{w,x\WJs,k}$, where 
\begin{equation} \label{eq:m''}
k'':=k-\delta_{j0}\pair{\vpi_{i}}{w^{-1}\alpha_{j}^{\vee}} + 
\delta_{j0}\pair{\vpi_{i}}{x^{-1}\alpha_{j}^{\vee}}; 
\end{equation}
here we use the fact that $\qwt(\ed(\bp) \Rightarrow x) \in \QJvp$ and 
$\qwt(\ed(\Phi(\bp)) \Rightarrow s_{j}x) \in \QJvp$ for $\bp \in \hbR{w,x\WJs,k}$ 
since $\ed(\bp),\,x \in x\WJs$ and $\ed(\Phi(\bp)),\,s_{j}x \in s_{j}x\WJs$, and then 
compute 
\begin{align*}
\pair{\vpi_{i}}{\qwt(\Phi(\bp))} & = 
\pair{\vpi_{i}}{\qwt(\Phi(\bp)) + \qwt(\ed(\Phi(\bp)) \Rightarrow s_{j}x)} \\
& = \pair{\vpi_{i}}{\qwt(\bp) + \qwt(\ed(\bp) \Rightarrow x) - \delta_{j0}w^{-1}\alpha_{j}^{\vee} + 
\delta_{j0}x^{-1}\alpha_{j}^{\vee}} \\
& = \pair{\vpi_{i}}{\qwt(\bp)} - \delta_{j0}\pair{\vpi_{i}}{w^{-1}\alpha_{j}^{\vee}} + 
\delta_{j0} \pair{\vpi_{i}}{x^{-1}\alpha_{j}^{\vee}} = k''. 
\end{align*}
Notice that $\# \hbR{s_{j}w,s_{j}x\WJs,k''} = \# \hbR{w,x\WJs,k} \ge 2$. 
By the induction hypothesis, $\# \hbR{s_{j}w,s_{j}x\WJs,k''} \in 2\BZ_{\ge 1}$, 
and hence $\# \hbR{w,x\WJs,k} \in 2\BZ_{\ge 1}$. Also, 
by using this bijection $\Phi$ and the induction hypothesis, 
we compute
\begin{equation*}
\sum_{ \bp \in \hbR{w,x\WJs,k} }
 (-1)^{\ell(\bp)} =
\sum_{ \bq \in \hbR{s_{j}w,s_{j}x\WJs,k''} }
 (-1)^{\ell(\bq)} \stackrel{ \text{(IH)} }{=} 0. 
\end{equation*}

\medskip

\paragraph{\bf Case 3'.}
%
Assume that $x^{-1}\alpha_{j} \in \Delta_{\J}$; 
note that $x\WJs = s_{j}x\WJs$, and that 
$y^{-1}\alpha_{j} \in \Delta_{\J}$ for all $y \in x\WJs$, 
which implies that $\pair{\vpi_{i}}{y^{-1}\alpha_{j}^{\vee}} = 0$. 
Here, for $w',x' \in W$ and $p \in \BZ_{\ge 0}$, we set 
\begin{equation*}
(\hbR{w',x'\WJs,p})^{\pm} := \bigl\{ 
\bp \in \hbR{w',x'\WJs,p} \mid 
(\ed(\bp))^{-1}\alpha_{j} \in \Delta^{\pm} \bigr\}. 
\end{equation*}

Now, as in Case 2 (see also Subcases 3.1 and 3.2)
of the proof of Theorem~\ref{thm:main3}, we deduce that 
there exists a bijection $\Phi:
(\hbR{w,x\WJs,k})^{+} \rightarrow (\hbR{s_{j}w,x\WJs,k'})^{-}$
such that $\ell(\Phi(\bp))=\ell(\bp)$; 
here we note that $k'=k''$ since $\pair{\vpi_{i}}{x^{-1}\alpha_{j}^{\vee}} = 0$. 
Also, in exactly the same way as in Case 1 (see also Subcases 3.1 and 3.2)
in the proof of Theorem~\ref{thm:main3}, we deduce that 
there exists a bijection $\Psi:
(\hbR{w,x\WJs,k})^{-} \rightarrow (\hbR{s_{j}w,x\WJs,k'})^{-}$
such that $\ell(\Psi(\bp))=\ell(\bp)-1$. 
Using these bijections $\Psi$ and $\Phi$ above, we compute
\begin{align*}
& \sum_{ \bp \in \hbR{w,x\WJs,k} }
  (-1)^{\ell(\bp)} =
\sum_{ \bp \in (\hbR{w,x\WJs,k})^{+} }
  (-1)^{\ell(\bp)} + 
\sum_{ \bp \in (\hbR{w,x\WJs,k})^{-} }
  (-1)^{\ell(\bp)} \\[3mm]
& = \sum_{ \bq \in (\hbR{s_{j}w,x\WJs,k'})^{-} }
  (-1)^{\ell(\bq)} -
\sum_{ \bq \in (\hbR{s_{j}w,x\WJs,k'})^{-} }
  (-1)^{\ell(\bq)} = 0. 
\end{align*}
Also, by this equality and the assumption that $\#\hbR{w,x\WJs,k} \ge 2$, 
we see that $\# \hbR{w,x\WJs,k} \in 2\BZ_{\ge 1}$.

This completes the proof of Theorem~\ref{thm:hR1even}.
\end{proof}

We set
\begin{equation}
D := \sum_{\beta \in \Delta^{+}_{2}} \beta^{\vee} + 
\sum_{\gamma \in \DJs} \gamma^{\vee} \in Q^{\vee,+},
\end{equation}
where $\Delta^{+}_{2}:=\bigl\{ \beta \in \Delta^{+} \mid 
\pair{\vpi_{i}}{\beta^{\vee}} = 2 \bigr\}$. 
Let $w,x \in W$, and $k \in \BZ_{\ge 0}$. We claim that 
\begin{equation} \label{eq:D}
\qwt(\bp_2) + \qwt(\ed(\bp_2) \Rightarrow x) \le D \quad 
\text{for all $\bp=(\bp_{2},\bp_{1}) \in \hbR{w,x\WJs,k}$}.
\end{equation}
Indeed, let $\bp=(\bp_{2},\bp_{1}) \in \hbR{w,x\WJs,k}$, 
and set $z:=\ed(\bp) = \ed(\bp_2)$. 
By the definition, $\bp_{2}$ is a label-increasing (and hence shortest) 
directed path from $w$ to $z$ whose labels are all contained in $\Delta^{+}_{2}$. 
Hence we have
\begin{equation} \label{eq:D1}
\qwt(\bp_2) \le \sum_{\beta \in \Delta^{+}_{2}} \beta^{\vee}. 
\end{equation}
Also, since $z \in x\WJs$, it follows from Lemma~\ref{lem:WJs} that 
the labels of a label-increasing (and hence shortest) directed path 
from $z$ to $x$ are all contained in $\DJs$. Hence we have
\begin{equation} \label{eq:D2}
\qwt(z \Rightarrow x) \le \sum_{\gamma \in \DJs} \gamma^{\vee}. 
\end{equation}
Combining \eqref{eq:D1} and \eqref{eq:D2}, 
we obtain \eqref{eq:D}, as desired. Write $D$ as 
$D=\sum_{j \in I} D_{j} \alpha_{j}^{\vee}$ 
with $D_{j} \in \BZ_{\ge 0}$ for $j \in I$. 
%
%
\begin{rem} \label{rem:hR1even}
By \eqref{eq:RvshR} and \eqref{eq:D}, 
in order to prove Conjecture~\ref{conj:R1even}, 
it suffices to show that $\# \bR{w,x,d} \in \{ 1 \} \sqcup 2\BZ_{\ge 0}$ 
for all $w,x \in W$ and $d \in Q^{\vee,+}$ such that $d \le D$; 
notice that $W$ and $\bigl\{ d \in Q^{\vee,+} \mid d \le D \bigr\}$ are finite sets. 
Indeed, let $d = \sum_{j \in I} d_{j} \alpha_{j}^{\vee}$ 
be an arbitrary element of $Q^{\vee,+}$. 
If $d_{i} > D_{i}$, then eqref{eq:D} and 
\eqref{eq:RvshR} imply that $\bR{w,x,d} \subset \hbR{w,x\WJs,d_{i}} = \emptyset$ 
for all $w,x \in W$. Assume that $d_{i} \le D_{i}$, and set
\begin{equation*}
\ti{d}:=\sum_{j \in I} \min \{ D_{j}, d_{j} \} \alpha_{j}^{\vee}; 
\end{equation*}
notice that $\ti{d} \le D$ and $\ti{d} \le d$. We claim that 
$\bR{w,x,d}=\bR{w,x,\ti{d}}$ for all $w,x \in W$. 
Since $\ti{d} \le d$ and $\min \{ D_{i}, d_{i} \} = d_{i}$ 
(and hence $\hbR{w,x\WJs,d_{i}} = \hbR{w,x\WJs,\min \{ D_{i}, d_{i} \}}$), 
we see by \eqref{eq:RvshR} that $\bR{w,x,d} \supset \bR{w,x,\ti{d}}$. 
Let $\bp \in \bR{w,x,d}$; 
note that $\bp \in \hbR{w,x\WJs,d_{i}} = \hbR{w,x\WJs,\min \{ D_{i}, d_{i} \}}$. 
Hence, by \eqref{eq:RvshR}, it suffices to show that $\delta \le \ti{d}$, 
where $\delta:=\qwt(\bp) + \qwt(\ed(\bp) \Rightarrow x)$. 
Write $\delta$ as $\delta = \sum_{j \in I} c_{j}\alpha_{j}^{\vee}$. 
By \eqref{eq:RvshR}, we have $\delta \le d$ since $\bp \in \bR{w,x,d}$. 
Also, by \eqref{eq:D}, we have $\delta \le D$. Therefore we get $\delta \le \ti{d}$, 
as desired. Thus we have proved the claim, which implies that 
$\# \bR{w,x,d}=\# \bR{w,x,\ti{d}}$. 
\end{rem}
%
%
\begin{cor} \label{cor:hR1even}
Assume that $\vpi_{i}$ is quasi-minuscule. 
Let $w,x \in W$, and let $d = \sum_{j \in I} d_{j}\alpha_{j}^{\vee} \in Q^{\vee,+}$. 
If $d_{j} \ge D_{j}$ for all $j \in \J=I \setminus \{i\}$, then 
$\bR{w,x,d} = \hbR{w,x\WJs,d_{i}}$. Therefore, by Theorem~\ref{thm:hR1even}, 
$\#\bR{w,x,d} \in \{1\} \sqcup 2\BZ_{\ge 0}$, 
and if $\#\bR{w,x,d} \ge 2$, then 
$\sum_{\bp \in \bR{w,x,d}} (-1)^{\ell(\bp)}=0$; 
in this case, 
$\langle \CO^{s_i}, \CO^{w}, \CO_{x} \rangle_{d} =  1$ 
by Corollary~\ref{cor:twp}. 
\end{cor}

\begin{proof}
In order to prove $\bR{w,x,d} = \hbR{w,x\WJs,d_{i}}$, 
it suffices to show that 
$\qwt(\bp_{2}) + \qwt ( \ed(\bp_{2}) \Rightarrow x) \le d$ 
for all $\bp=(\bp_{2},\bp_{1}) \in \hbR{w,x\WJs,d_{i}}$; 
see \eqref{eq:RvshR}. Let $\bp=(\bp_{2},\bp_{1}) \in \hbR{w,x\WJs,d_{i}}$, 
and write $\delta:=\qwt(\bp_{2}) + \qwt ( \ed(\bp_{2}) \Rightarrow x)$ as 
$\delta = \sum_{j \in I} c_{j}\alpha_{j}^{\vee}$. 
By \eqref{eq:D}, we see that $c_{j} \le D_{j} \le d_{j}$ 
for all $j \in \J=I \setminus \{i\}$. It remains to show that 
$c_{i} \le d_{i}$. Since $\qwt ( \ed(\bp_{2}) \Rightarrow x) \in \QJvp$, 
and since $\bp=(\bp_{2},\bp_{1}) \in \hbR{w,x\WJs,d_{i}}$, 
it follows that
\begin{equation*}
\pair{\vpi_{i}}{\delta} = \pair{\vpi_{i}}{\qwt(\bp_{2})} = d_{i}. 
\end{equation*}
Also, we see that $\pair{\vpi_{i}}{\delta} = c_{i}$. 
Thus we get $c_{i} = d_{i}$. Therefore we conclude that 
$\qwt(\bp_{2}) + \qwt ( \ed(\bp_{2}) \Rightarrow x) \le d$, 
which implies that $\bR{w,x,d} = \hbR{w,x\WJs,d_{i}}$. 
This completes the proof of the corollary. 
\end{proof}

\appendix

%
\section{Examples.}
\label{sec:example}

In Appendix~\ref{sec:example}, 
we assume that $\Fg$ is of type $B_{n}$ with $I=\{1,2,\dots,n\}$,
where our numbering of the nodes of the Dynkin diagram is 
the same as that in \cite[Section 11.4]{H} 
($\alpha_{n}$ is the unique short simple root). 
We recall that all the fundamental weights $\vpi_{i}$ are quasi-minuscule, 
and hence we can take $N=N_{i}=2$. 
We write an element $\bp=(\bp_{2},\bp_{1}) \in \bR{w,x,d}$ 
simply as $\bp_{2}$; recall that $\bp_{1} = \bt_{\ed(\bp_2)}$. 
Also, for an edge $x \edge{\alpha} y$ in $\QBG(W)$, 
we write $x \Be{\alpha} y$ (resp., $x \Qe{\alpha} y$) 
to indicate that the edge is a Bruhat (resp., quantum) edge. 
%
%
\begin{ex} \label{ex:i2}
Assume that $n \ge 3$ and $i = 2$. 
Recall that $\J=I \setminus \{2\}$. 
In this case, we deduce that $s_{\theta} = \mcr{\lng}^{\J} \in \WJ$, 
with $\lng$ the longest element of $W$, and that 
$s_{\theta}s_{j} = s_{j} s_{\theta}$ for all $j \in \J$, where 
$\theta = \alpha_{1}+2\alpha_{2}+\cdots+2\alpha_{n} \in \Delta^{+}$ is the highest root; 
note that $\theta = \vpi_{2}$. 

\begin{claim} \mbox{} \label{c:exa}
Let $d=\sum_{j \in I} d_{j}\alpha_{j}^{\vee} \in Q^{\vee,+}$ be such that $d_{2}=2$, 
and let $w,x \in W$. 
\begin{enu}
\item If $w \notin s_{\theta}\WJs$ or $x \notin \WJs$, then 
$\bR{w,x,d} = \emptyset$. 

\item Assume that $w \in s_{\theta}\WJs$ and $x \in \WJs$. 
Write $w$ as $w = s_{\theta}v$ with $v \in \WJs$. Then,
\begin{equation}
\bR{w,x,d} = 
\begin{cases}
\Bigl\{ w = s_{\theta}v = v s_{\theta} \Qe{\theta} v \Bigr\}
& \text{if $d \ge \theta^{\vee}+\qwt(v \Rightarrow x)$}, \\
\emptyset & \text{otherwise}. 
\end{cases}
\end{equation}
\end{enu}
\end{claim}

\noindent
{\it Proof of Claim~\ref{c:exa}.}
We set $\Delta^{+}_{2}:=\bigl\{ \beta \in \Delta^{+} \mid \pair{\vpi_{2}}{\beta^{\vee}} = 2 \bigr\}$; 
notice that 
\begin{equation*} 
\Delta^{+}_{2}=\bigl\{ \beta_1 := \alpha_1+\cdots+\alpha_{n},\,
\beta_{2} := \alpha_{2}+\cdots+\alpha_{n}, \, \theta \bigr\}, 
\end{equation*}
where $\beta_{1}$ and $\beta_{2}$ are short roots, and 
$\theta$ is a long root. In order to prove the claim above, 
it suffices to show the following: 
\begin{enu}
\item[(a)] If $y \Qe{\beta} z$ for some 
$y,z \in W$ and $\beta \in \Delta^{+}_{2}$, 
then $\beta = \theta$ and $y \in s_{\theta}\WJs$ and $z \in \WJs$. 

\item[(b)] For each $v \in \WJs$, there exists a quantum edge $s_{\theta}v \Qe{\theta} v$. 

\item[(c)] There does not exist an edge of the form 
$y' \edge{\beta} y$ for any $\beta \in \Delta^{+}_{2}$ and $y \in s_{\theta}\WJs$, $y' \in W$.
Also, there does not exist an edge of the form 
$z \edge{\beta} z'$ for any $\beta \in \Delta^{+}_{2}$ and $z \in \WJs$, $z' \in W$. 
\end{enu}
First, let us show part (a). Recall that $\ell(s_{\gamma}) \le 2 \pair{\rho}{\gamma^{\vee}} -1$ 
for all $\gamma \in \Delta^{+}$; see, e.g., \cite[Lemma 4.1]{LNSSS1}. 
We call $\gamma \in \Delta^{+}$ a quantum root if 
$\ell(s_{\gamma}) = 2 \pair{\rho}{\gamma^{\vee}} -1$. 
Since $\ell(u s_{\gamma}) \ge \ell(u) - \ell(s_{\gamma})$, we deduce that 
if $u \Qe{\gamma} us_{\gamma}$, then $\gamma$ is a quantum root. 
Now, assume that $y \Qe{\beta} z$. Because neither $\beta_{1}$ nor $\beta_{2}$ 
is a quantum root (see, e.g., \cite[Lemma~4.2]{LNSSS1}), 
it follows that $\beta \ne \beta_1, \beta_{2}$, and hence $\beta = \theta$. 
Write $y$ as $y = v s_{\theta}$ 
for some $v \in W$ such that $\ell(y) = \ell(v) + \ell(s_{\theta})$. 
Note that $\ell(s_{\theta}v^{-1}) = \ell(s_{\theta}) + \ell(v^{-1})$. 
Since $s_{\theta}v^{-1} \ge s_{\theta}$, we have 
$\mcr{s_{\theta}v^{-1}} \ge \mcr{s_{\theta}} = \mcr{\lng}$. 
Hence we get $\mcr{s_{\theta}v^{-1}} = \mcr{\lng} = s_{\theta}$, which implies that 
$s_{\theta}v^{-1} \in s_{\theta}\WJs$. Therefore, $v \in \WJs$. 
Because $s_{\theta}s_{j} = s_{j} s_{\theta}$ for all $j \in \J$, 
we see that $y = vs_{\theta} = s_{\theta}v \in s_{\theta}\WJs$, 
and also $z = ys_{\theta} = v \in \WJs$. Thus we have shown part (a). 

Next, let us show part (b). Let $v \in \WJs$. 
As seen above, we have $s_{\theta}v = vs_{\theta}$. 
Also, since $s_{\theta} \in \WJ$, it follows that 
$\ell(vs_{\theta}) = \ell(s_{\theta}v) = \ell(s_{\theta}) + \ell(v)$, 
and hence $\ell(v) = \ell(vs_{\theta}) - \ell(s_{\theta}) = 
\ell(vs_{\theta}) - 2 \pair{\rho}{\theta^{\vee}} + 1$; recall that $\theta$ is a quantum root. 
Therefore there exists a quantum edge $s_{\theta}v \Qe{\theta} v$. 
Thus we have shown part (b). 

Finally, let us show part (c). Suppose, for a contradiction, that 
there exists an edge $z \edge{\beta} z'$ for some $\beta \in \Delta^{+}_{2}$ and 
$z \in \WJs$, $z' \in W$. Note that $\mcr{z'} \ne \mcr{z} = e$ since $\beta \in \DJp$. 
By \cite[Lemma~6.2]{LNSSS2}, we deduce that $(\mcr{z'},\mcr{z}) = (\mcr{z'},e) \in \QLS(\vpi_{2})$. 
However, this contradicts \cite[Lemma 3.3]{MNS1} since $\mcr{z'} \ne e$. 
Similarly, suppose, for a contradiction, that 
there exists an edge $y' \edge{\beta} y$ for some $\beta \in \Delta^{+}_{2}$ and 
$y \in s_{\theta}\WJs$, $y' \in W$. Note that $\mcr{y'} \ne \mcr{y} = \mcr{\lng}$ since $\beta \in \DJp$. 
By \cite[Lemma~6.2]{LNSSS2}, we deduce that $(\mcr{y},\mcr{y'}) = (\mcr{\lng},\mcr{y'}) \in \QLS(\vpi_{2})$. 
However, this contradicts (the dual version of) \cite[Lemma 3.3]{MNS1}; see also \cite[Section 4.5]{LNSSS2}. 
Thus we have shown part (c). \bqed
\end{ex}

%
\begin{ex} \label{ex:b4i3a}
Assume that $n = 4$ and $i = 3$. 
Let $w = s_3s_4s_2s_3s_4s_2s_3$, and $x = e$. 
We have the following directed paths $\bp_{1}$ and $\bp_{2}$: 
\begin{equation*}
\bp_{1} : w \Be{\alpha_1+\alpha_2+2\alpha_3+2\alpha_4}
s_3s_4s_1s_2s_3s_4s_2s_3 \Qe{\alpha_2+2\alpha_3+2\alpha_4} s_1 \in x\WJs, 
\end{equation*}
\begin{equation*}
\bp_{2} : w \Qe{\alpha_2+2\alpha_3+2\alpha_4} e = \tbmin{x}{\J}{w} \in x\WJs; 
\end{equation*}
note that $\ell(\bp_{1})=\ell(\bp_{2})+1$, 
$\qwt(\bp_{1}) = \qwt(\bp_{2}) = (\alpha_2+2\alpha_3+2\alpha_4)^{\vee} = 
\alpha_{2}^{\vee}+2\alpha_{3}^{\vee}+\alpha_{4}^{\vee}$, 
and $\qwt(\ed(\bp_{1}) \Rightarrow x) = \alpha_{1}^{\vee}$, 
$\qwt(\ed(\bp_{2}) \Rightarrow x) = 0$. We can check that 
\begin{equation*}
\bR{w,x,d} = 
  \begin{cases}
  \bigl\{ \bp_{2} \bigr\} & 
  \text{if $d = d_{2} \alpha_{2}^{\vee}+2\alpha_{3}^{\vee}+ d_{4}\alpha_{4}^{\vee}$
  with $d_{2},d_{4} \in \BZ_{\ge 1}$}, \\
  \bigl\{ \bp_{1}, \bp_{2} \bigr\} & 
  \text{if $d = d_{1}\alpha_{1}^{\vee} + d_{2} \alpha_{2}^{\vee}+2\alpha_{3}^{\vee} + d_{4}\alpha_{4}^{\vee}$
  with $d_{1},d_{2},d_{4} \in \BZ_{\ge 1}$}, \\
  \emptyset & \text{otherwise}. 
  \end{cases}
\end{equation*}
\end{ex}
%
%
\begin{ex} \label{ex:b4i3b}
Assume that $n = 4$ and $i = 3$. 
Let $w = s_3s_4s_2s_3s_4s_1s_2s_3s_4s_1s_2s_3s_1$, 
and $x = s_2 s_3 s_4 s_3$. 
We have the following directed paths $\bp_{1}$ and $\bp_{2}$: 
\begin{equation*}
\bp_{1} : w \Qe{\alpha_2+2\alpha_3+2\alpha_4} \underbrace{s_2s_3s_4s_3}_{=x}s_2s_1 
= \tbmin{x}{\J}{w} \in x \WJs, 
\end{equation*}
\begin{equation*}
\bp_{2} : w \Qe{\alpha_1+\alpha_2+2\alpha_3+2\alpha_4} s_3s_4s_3s_2
\Be{\alpha_2+2\alpha_3+2\alpha_4} \underbrace{s_2s_3s_4s_3}_{=x}s_2 \in x \WJs; 
\end{equation*}
note that 
\begin{align*}
& \ell(\bp_{1})=\ell(\bp_{2})-1, \\
& \qwt(\bp_{1}) = (\alpha_2+2\alpha_3+2\alpha_4)^{\vee} = 
\alpha_2^{\vee} + 2\alpha_3 ^{\vee} + \alpha_4^{\vee}, \qquad 
\qwt(\ed(\bp_1) \Rightarrow x) = \alpha_1^{\vee}+\alpha_2^{\vee}, \\
& \qwt(\bp_{2}) = (\alpha_1+\alpha_2+2\alpha_3+2\alpha_4)^{\vee} = 
\alpha_1^{\vee} + \alpha_2^{\vee} + 2\alpha_3 ^{\vee} + \alpha_4^{\vee}, \qquad 
\qwt(\ed(\bp_2) \Rightarrow x) = \alpha_2^{\vee}.
\end{align*}
We can check that 
\begin{equation*}
\bR{w,x,d} = 
  \begin{cases}
  \bigl\{ \bp_{1}, \bp_{2} \bigr\} 
  & \text{if $d = d_{1}\alpha_{1}^{\vee} + d_{2} \alpha_{2}^{\vee}+2\alpha_{3}^{\vee} + d_{4}\alpha_{4}^{\vee}$
  with $d_{1},d_{4} \in \BZ_{\ge 1}$ and $d_{2} \in \BZ_{\ge 2}$}, \\
  \emptyset & \text{otherwise}. 
  \end{cases}
\end{equation*}
\end{ex}
%
%
\begin{ex} \label{ex:b4i3c}
Assume that $n = 4$ and $i = 3$. 
Let $w = s_3s_4s_2s_3s_4s_1s_2s_3s_4s_1s_2s_3s_1s_2s_1$, 
and $x = s_1s_2s_3s_4s_3s_1s_2s_1$. 
We have the following directed paths $\bp_{k}$, $1 \le k \le 4$: 
\begin{equation*}
\bp_{1} : w \Qe{\alpha_2+2\alpha_3+2\alpha_4} s_1s_2s_3s_4s_3s_1s_2s_1 = \tbmin{x}{\J}{w} \in x \WJs, 
\end{equation*}
\begin{equation*}
\bp_{2} : w \Qe{\alpha_1+\alpha_2+2\alpha_3+2\alpha_4} s_2s_3s_4s_3s_1s_2
\Be{\alpha_2+2\alpha_3+2\alpha_4} s_1s_2s_3s_4s_3s_1s_2
 \in x\WJs, 
\end{equation*}
\begin{equation*}
\bp_{3} : w \Qe{\alpha_1+2\alpha_2+2\alpha_3+2\alpha_4} s_3s_4s_3s_1
\Be{\alpha_2+2\alpha_3+2\alpha_4} s_1s_2s_3s_4s_3
 \in x\WJs, 
\end{equation*}
\begin{equation*}
\bp_{4} : w \Qe{\alpha_1+2\alpha_2+2\alpha_3+2\alpha_4} s_3s_4s_3s_1
\Be{\alpha_1+\alpha_2+2\alpha_3+2\alpha_4} s_2s_3s_4s_3s_1
\Be{\alpha_2+2\alpha_3+2\alpha_4} s_1s_2s_3s_4s_3s_1
 \in x\WJs; 
\end{equation*}
note that
\begin{align*}
& \ell(\bp_{2}) = \ell(\bp_{3}) = \ell(\bp_1)+1, \quad 
  \ell(\bp_{4}) = \ell(\bp_1)+2, \\
& \qwt(\bp_{1}) = \alpha_{2}^{\vee}+2\alpha_{3}^{\vee}+\alpha_{4}^{\vee}, \quad 
  \qwt(\ed(\bp_1) \Rightarrow x) = 0, \\
& \qwt(\bp_{2}) = \alpha_{1}^{\vee}+\alpha_{2}^{\vee}+2\alpha_{3}^{\vee}+\alpha_{4}^{\vee}, \quad 
  \qwt(\ed(\bp_2) \Rightarrow x) = 0, \\
& \qwt(\bp_{3}) = \alpha_{1}^{\vee}+2\alpha_{2}^{\vee}+2\alpha_{3}^{\vee}+\alpha_{4}^{\vee}, \quad 
  \qwt(\ed(\bp_3) \Rightarrow x) = 0, \\
& \qwt(\bp_{4}) = \alpha_{1}^{\vee}+2\alpha_{2}^{\vee}+2\alpha_{3}^{\vee}+\alpha_{4}^{\vee}, \quad 
  \qwt(\ed(\bp_4) \Rightarrow x) = 0.
\end{align*}
We can check that 
{\small%
\begin{equation*}
\bR{w,x,d} = 
\begin{cases}
\bigl\{ \bp_{1}, \bp_{2}, \bp_{3}, \bp_{4} \bigr\}
&\text{if $d = d_{1}\alpha_{1}^{\vee} + d_{2} \alpha_{2}^{\vee}+2\alpha_{3}^{\vee} + d_{4}\alpha_{4}^{\vee}$
  with $d_{1},d_{4} \in \BZ_{\ge 1}$ and $d_{2} \in \BZ_{\ge 2}$}, \\
\bigl\{ \bp_{1}, \bp_{2} \bigr\} & 
\text{if $d = d_{1}\alpha_{1}^{\vee} + \alpha_{2}^{\vee}+2\alpha_{3}^{\vee} + d_{4}\alpha_{4}^{\vee}$
with $d_{1},d_{4} \in \BZ_{\ge 1}$}, \\
\bigl\{ \bp_{1} \bigr\} & 
\text{if $d = d_{2}\alpha_{2}^{\vee}+2\alpha_{3}^{\vee} + d_{4}\alpha_{4}^{\vee}$
with $d_{2},d_{4} \in \BZ_{\ge 1}$}, \\
\emptyset & \text{otherwise}. 
\end{cases}
\end{equation*} }
\end{ex}

%

\end{document}